\documentclass[a4paper,french,english,10pt]{article}
\usepackage{a4wide}
\usepackage[T1]{fontenc}
\usepackage[latin9]{inputenc}
\usepackage{babel}
\usepackage{xcolor}
\usepackage{amsmath}
\usepackage{graphicx}
\usepackage{graphics}
\usepackage{epsfig}
\usepackage{color}
\usepackage{amsfonts}
\usepackage{amssymb,latexsym}
\usepackage{amscd}
\usepackage{multirow}
\usepackage{amsthm}
\usepackage[pdftex,bookmarks=true,bookmarksopen=true,colorlinks=true,linkbordercolor=white, citecolor=blue, linkcolor=red]{hyperref}

\usepackage{comment}

\newcommand{\dr}{\partial_R}
\newcommand{\ds}{\displaystyle}
\newcommand{\dz}{\partial_Z}
\newcommand{\dphi}{\partial_{\phi}}
\newcommand{\ephi}{\mathbf{e}_{\phi}}
\newcommand{\dt}{\partial_t}
\newcommand{\lp}{\triangle_{pol}}

\newcommand{\gradp}{\nabla_{pol}}

\newcommand{\gs}{\triangle^{*}}
\newcommand{\U}{\mathbf{U}}

\newcommand{\vv}{\mathbf{v}}
\newcommand{\vB}{\mathbf{B}}
\newcommand{\vJ}{\mathbf{J}}
\newcommand{\vE}{\mathbf{E}}
\newcommand{\rot}{\nabla \times}
\newcommand{\er}{\mathbf{e}_R}
\newcommand{\ez}{\mathbf{e}_Z}
\newcommand{\eps}{\varepsilon}
\newcommand\numberthis{\addtocounter{equation}{1}\tag{\theequation}}
\newtheorem{theorem}{Theorem}[section]
\newtheorem{lemma}[theorem]{Lemma}
\newtheorem{proposition}[theorem]{Proposition}
\newtheorem{corollary}[theorem]{Corollary}

\newenvironment{pr oof}[1][D\'emonstration]{\begin{trivlist}
\item[\hskip \labelsep {\bfseries #1}]}{\end{trivlist}}

\begin{document}

\title{Energy conservation and numerical stability for the reduced MHD models of the non-linear JOREK code}

\author{Emmanuel Franck\footnotemark[1] \footnotemark[2],
Matthias H\"{o}lzl\footnotemark[1], 
Alexander Lessig\footnotemark[1], 
Eric Sonnendr\"{u}cker\footnotemark[1]
}
 \footnotetext[1]{Max-Planck-Institut f\"ur Plasmaphysik, Boltzmannstra\ss e 2
D-85748 Garching, Germany.}
\footnotetext[2]{Email adress: emmanuel.franck@ipp.mpg.de}

\maketitle

\begin{abstract} 
In this paper we present a rigorous derivation of the reduced MHD models 
with and without parallel velocity that are implemented in the 
non-linear MHD code JOREK. The model we obtain contains some terms that 
have been neglected in the implementation but might be relevant in the 
non-linear phase. These are necessary to guarantee exact conservation 
with respect to the full MHD energy.

For the second part of this work, we have replaced the linearized time 
stepping of JOREK by a non-linear solver based on the Inexact Newton 
method including adaptive time stepping. We demonstrate that this 
approach is more robust especially with respect to numerical errors in 
the saturation phase of an instability and allows to use larger time 
steps in the non-linear phase.
\end{abstract}

\tableofcontents

\section{Physical context and resistive MHD}
\subsection{Physical context: ITER and ELM's simulations}
The aim of magnetic confinement fusion is to develop a power plant that 
gains energy from the fusion of deuterium and tritium in a magnetically confined plasma. ITER, a tokamak type fusion 
experiment currently being built in the South of France, is the next 
step towards this goal.

In order to achieve a reasonable lifetime of first wall materials in 
ITER and future fusion reactors, plasma instabilities like edge 
localized modes (ELMs) \cite{zohm,snyder} need to be well controlled. Numerical modelling 
can help to develop the necessary understanding of the relevant physical 
processes. A physical model well suited to describe those large scale 
instabilities is the set of magneto-hydrodynamic equations (MHD) or the 
simpler reduced MHD model.

\subsection{Resistive MHD}
We begin by introducing the resistive Magnetohydrodynamic (MHD) fluid system in 3D. The spatial variable is $\mathbf{x}\in\mathbb{R}^3$. We note $\rho$ the mass density of the plasma, $\mathbf{v}$ the velocity, $T$ the temperature, $p=\rho T$ the pressure, $\mathbf{B}$ the magnetic field, $\mathbf{J}$ the current and $\mathbf{E}$ the electric field. The evolution of the plasma can be described by the following MHD model
\begin{equation}\label{MHD}
\left\{\begin{array}{l}
\ds\dt \rho+\nabla\cdot(\rho\vv)=0\\
\\
\ds\rho\dt \vv+\rho\vv\cdot\nabla \vv+\nabla (p)=\vJ\times\vB+ \nabla \cdot(\nu \nabla \vv) \\
\\
\ds \dt p+\vv\cdot\nabla p+\gamma p\nabla\cdot \vv=0\\
\\
\ds\dt \vB=-\rot \vE=\rot (\vv\times\vB-\eta\vJ)\\
\\
\ds \nabla \times \mathbf{B} =\mathbf{J}\\
\\
\ds \nabla\cdot\mathbf{B}=0
\end{array}\right.
\end{equation}
with $\nu$ the viscosity coefficient, $\eta$ the resistivity coefficient. The resistive term originates from the collision between the two species electrons and ions present in the plasma. The viscosity term is a very simple approximation of the stress tensor. The resistive MHD model used here is a simplification of two fluids models (extended MHD). The numerical properties of extended MHD terms in JOREK are beyond the scope of this paper and will be investigated in a future publication.
First we recall the  energy conservation and dissipation properties of the resistive MHD model.
\begin{proposition}
The total energy of the MHD model is given by the sum of the kinetic energy, magnetic energy and internal energy:
$$
E=\rho \frac{|\vv|^2}{2}+\frac{|\vB|^2}{2}+\frac{1}{\gamma-1}p.
$$
with $p=\rho T$ and $\gamma =\frac53$. 
The balance law for the total energy is given by
$$
\dt E+\nabla \cdot\left[\vv \left(\rho\frac{|\vv|^2}{2}+\frac{\gamma}{\gamma-1}p\right)-(\vv\times\vB)\times\vB+\eta (\vJ\times \vB)\right]=-\eta \mid\vJ\mid^2+(\nabla \cdot(\nu \nabla \vv) )\cdot \vv)
$$
If $\vB=\vv=\mathbf{0}$ and $\rho=T=0$ on $\partial \Omega$ we obtain
$$
\frac{d}{dt}\int_{\Omega} E =-\eta \int_{\Omega} \mid \vJ^2\mid - \nu\int_{\Omega} \mid \mathbf{W}\mid ^2 - \nu \int_{\Omega} \mid \nabla \cdot \vv\mid ^2\leq 0
$$
with $\mathbf{W}=\nabla \times \vv$.
\end{proposition}
\begin{proof}
We multiply the mass equation by $\frac{|\vv|^2}{2}$, the velocity equation by $\vv$, the pressure equation by $\frac{1}{\gamma-1}$ and the magnetic field equation by $\vB$. We obtain 
\begin{equation}\label{MHDGOLmul}
\left\{\begin{array}{l}
\ds \frac{|\vv|^2}{2}\left(\dt \rho+\nabla\cdot(\rho\vv)\right)=0\\
\\
\ds\vv \cdot \left(\rho\dt \vv+\rho\vv\cdot\nabla \vv+\nabla p-\vJ\times\vB- \nabla \cdot(\nu \nabla \vv) \right)=0 \\
\\
\ds \dt \frac{p}{\gamma -1}+\frac{1}{\gamma -1}\vv\cdot\nabla p+\frac{\gamma}{\gamma -1} p\nabla\cdot \vv=0\\
\\
\ds\vB\cdot\dt \vB=\vB \cdot\left(\nabla \times \left(\vv\times\vB-\eta\vJ\right)\right)\\
\\
\ds \nabla \times \mathbf{B} =\mathbf{J}\\
\\
\ds \nabla\cdot\mathbf{B}=0
\end{array}\right.
\end{equation}
First we multiply the velocity equation  by $\vv$ and the mass equation by $\frac{|\vv|^2}{2}$ to obtain the following equation on the kinetic energy
$$
\dt (\rho \vv)+\nabla \cdot\left(\rho |\vv|^2\vv\right)+\vv\cdot\nabla p=(\vJ\times\vB)\cdot\vv + \nabla \cdot(\nu \nabla \vv).
$$
Adding this equation to the pressure equation and the magnetic field equation multiplied by $\vB$, we obtain 
\begin{align*}
\dt E= &-\frac{|\vv|^2}{2}\nabla\cdot(\rho\vv)-\rho|\vv|^2\cdot\nabla \vv-\frac{\gamma}{\gamma -1}\vv\cdot\nabla p-\frac{\gamma}{\gamma -1}p
\nabla \cdot \vv+\vB(\nabla\times\cdot(\vv \times \vB))+\vv\cdot(\vJ\times\vB)\\
& -\eta (\nabla\times \vJ\cdot \vB)+( \nabla \cdot(\nu \nabla \vv)\cdot \vv).
\end{align*}
Rearranging the terms this becomes
\begin{multline*}
\dt E= -\nabla\cdot\left((\rho\frac{|\vv|^2}{2}+\frac{\gamma}{\gamma-1}p)\vv\right)+\nabla\cdot\left((\vv\times\vB)\times\vB\right)
-\eta (\nabla\times \vJ\cdot \vB)+( \nabla \cdot(\nu \nabla \vv)\cdot\vv).
\end{multline*}
To obtain this we have used $\nabla \cdot(\mathbf{a}\times\mathbf{b})=(\nabla\times\mathbf{a})\cdot\mathbf{b}-\mathbf{a}(\nabla\times\mathbf{b})$. 
Now we use 
$$
\nabla\times \vJ\cdot \vB=\nabla\cdot(\vJ\times \vB)+|\vJ|^2.
$$
To finish the proof we use the definition of the vector Laplacian 
$$ \nabla \cdot( \nabla \vv) =\triangle \vv = \nabla (\nabla \cdot \vv)-\nabla\times(\nabla\times \vv)$$  and an integration by parts.
\end{proof}
\begin{corollary}
If the resistivity and viscosity coefficients are equal to zero the total energy is conserved in time and otherwise it is dissipated in time.
\end{corollary}
This result comes from the flux divergence theorem and the assumptions on the boundary conditions. Normally the dissipation introduced by the resistive and viscous terms is balanced by the viscous and the Ohmic heating to obtain at the end the conservation of the total energy. However it is classical to neglect these terms and work with the dissipative resistive MHD system. In the following, we will derive a reduced model with the same dissipative energy (or a really close energy). Indeed, energy conservation or dissipation is important for the numerical stability. 
\subsection{Reduced MHD models}
The reduced resistive MHD models are designed to reduce the CPU cost by making assumptions, which are reasonable for the tokamak configuration. Since the perturbation of the toroidal magnetic field is of second order (in terms of a small expansion parameter) and enters into the equation of motion only at third order, it can be neglected in the reduced MHD limit \cite{strauss} such that we take the toroidal magnetic field to be constant in time. The magnetic field $\vB$ can be split into two parts: the toroidal part $\vB_{\phi}$ and the poloidal part $\vB_{pol}$ given by
\begin{equation}\label{BBB}
\vB_{\phi}=\frac{F_0}{R}\ephi\mbox{ and  }\vB_{pol}=\frac{1}{R}\nabla \psi \times \ephi.
\end{equation}
The velocity field depends on the electrical potential in the poloidal plane and the parallel velocity (parallel to the magnetic field). It is given by
\begin{equation}\label{vvv}
\vv=\vv_{pol}+\vv_{||}=-R\nabla u \times \ephi+v_{||}\vB.
\end{equation}
This choice come from to the choice of the electrical potential $V=F_0 u$ with $E=\nabla V$ and the fact that the poloidal velocity is homogeneous to $\vE\times \vB$.
This potential formulation allows to reduce the number of variables and filter the fast magnetosonic waves of the MHD for nearly incompressible flows. The full MHD system with all waves is a very stiff problem with restrictive CFL stability conditions and bad conditioning for the numerical methods. Consequently eliminating these waves allows to obtain a less stiff problem, which is easier to solve. 
To obtain the final reduced models we plug the potential formulations in the full MHD model and use projections to obtain the equations on $u$ and $v_{||}$. For the equation on the electric potential  we project by applying the operator $\ephi\cdot\rot (R^2 ....)$ 
to the momentum equation. To obtain the equation on $v_{||}$ we project by applying the operator $\vB \cdot (...) $ 
to the momentum equation. 

One of the aims of this work is to derive exactly the reduced MHD model used in the JOREK  code and prove that this model satisfies the energy conservation law. Indeed the energy conservation is a very important property to ensure the numerical stability of the time evolution method for nonlinear models. 

\subsection{JOREK code}

The non-linear JOREK code was originally developed by Huysmans \cite{huymans,huymans2}, see also \cite{holzl,jorekphy1,jorekphy2,jorekphy3,jorekphy4,jorekphy5,jorekphy6,jorekphy7,jorekphy8,jorekphy9}, solves the reduced or full MHD equations in realistic 
three-dimensional tokamak geometry. The spatial discretization is performed  
by isoparametric B\'ezier finite elements in the poloidal plane and a 
toroidal Fourier decomposition. As a first step in a simulation, the 
Grad Shafranov equation given by
$$
\gs \psi = -R^2\frac{\partial p}{\partial \psi}-F\frac{\partial F}{\partial \psi}
$$
with $\gs \psi = R \frac{\partial }{\partial R}\left(\frac{1}{R}\frac{\partial \psi}{\partial R}\right)+\frac{\partial^2 \psi}{\partial Z^2}$, $F=R\vB_{\phi}$, $p$ the pressure and $\vB_{\phi}$ the toroidal magnetic field, 
is solved on an initial grid (Fig. 1, on the left) to 
calculate the plasma equilibrium and again on a grid aligned to the 
equilibrium magnetic flux surfaces (Fig. 1, on the right in blue). This second 
grid is used during the following time integration as well, in which the 
(reduced) MHD equations are solved by a fully implicit method 
(Crank-Nicholson or Gear scheme). The resulting large sparse matrix system is 
solved using the iterative GMRES method with a physics-based 
preconditioning during which the direct sparse matrix solver Pastix is 
employed. JOREK is implemented in Fortran 90/95 and uses a hybrid MPI 
plus OpenMP parallelization suitable for large scale simulations on supercomputers. The realistic treatment of the tokamak geometry 
including the plasma region, separatrix and X-point, as well as 
scrape-off layer and divertor region makes the code suitable for 
simulations of many different types of plasma instabilities.
\begin{figure}[t]
\begin{center}
    \includegraphics[scale=.45]{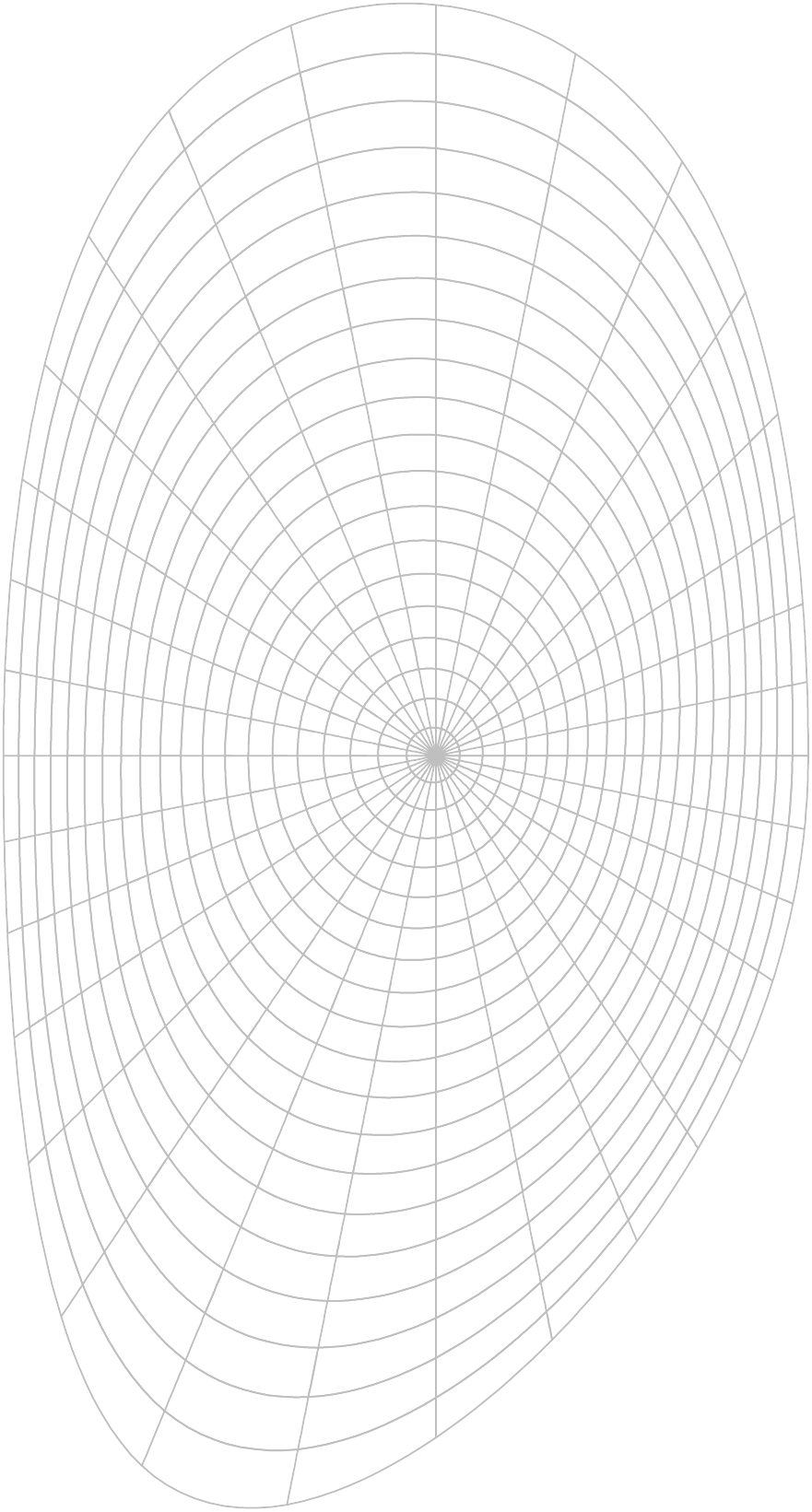}\hspace{60pt}\includegraphics[scale=.45]{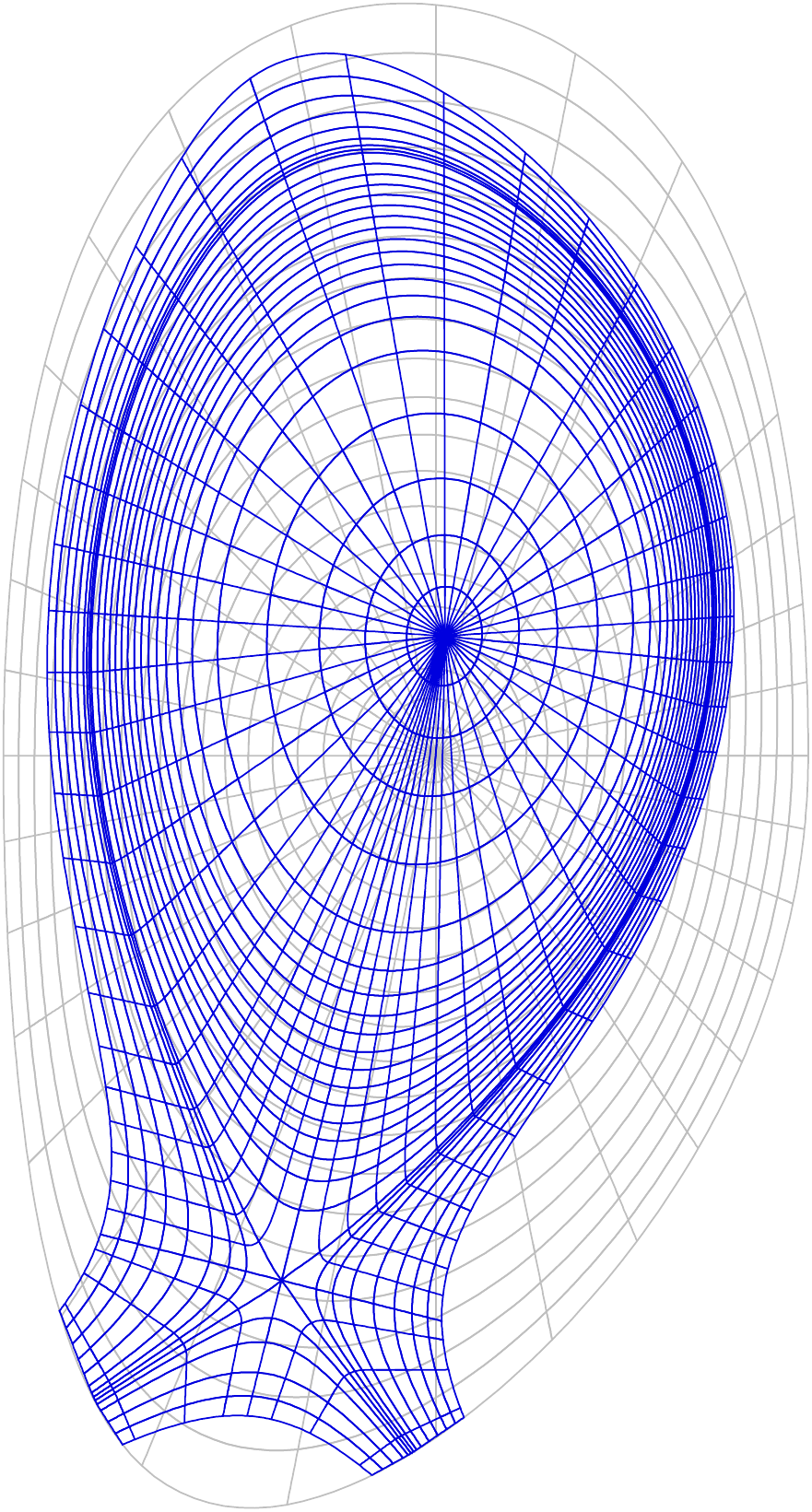}
        \caption{Initial grid (grey) and flux aligned grid (blue)  used in JOREK X-point simulations  (shown with reduced resolutions).}
\end{center}
\end{figure}
~\\
~\\
In the following we will provide a rigorous algebraic derivation of the 
reduced MHD equations that are implemented in JOREK from the full MHD 
equations (Sections 2.1--2.3) and investigate the energy conservation 
properties of this reduced MHD model (Section 2.4). In Section 3, we will 
 introduce a non-linear time integrator based on 
inexact Newton iterations for JOREK in order to increase the robustness 
and performance of the code in highly non-linear stages. Numerical tests 
of the non-linear time stepping scheme are presented in Section 4 and brief 
conclusions of the work are provided in Section 5.

\section{Derivation of the models}

The derivation of  reduced MHD models is not a new research topic. We can find some derivations of models with parallel velocity for small curvature in the Tokamak in \cite{strauss,kruger2}. These derivations are based on an asymptotic analysis with the small parameter $\eps$ which corresponds to the curvature of the geometry. In these calculations, some terms are neglected in the final models. In our case we use an algebraic derivation. Using the same assumptions for the magnetic field and the velocity field as in \cite{strauss,kruger2}.
The same method and the same type of calculation can be found in the works of R. Sart and B. Despr\'es in \cite{redMHD,redMHD2,rmhdinsta}. In these papers the authors propose two methods to obtain the reduced MHD in the low $\beta$ case, where $\beta$ is the ratio between plasma and magnetic pressures (which correspond to $p<< |\vB_{\phi}|^2$) for general density profiles. In this work we use the same technique as in their first paper, but we apply this method to obtain the more complicated models,  which are actually implemented in the JOREK code. 
So far, no exact derivation for the reduced MHD models implemented in JOREK had been published. For this reason we give these proofs and identify previously neglected terms in the reduced MHD physics models.
\subsection{Notation}
The fundamental coordinate system used in JOREK is the cylindrical system $(R,\phi,Z)$ illustrated in Fig. 2. The connection to cartesian coordinates is given by
\begin{equation}\label{coordinates}
\left\{\begin{array}{l}
X=R\cos\phi\\
Y=-R\sin\phi\\
Z=Z
\end{array}\right.
\end{equation}
We define $\er=\nabla R$, $\frac{1}{R}\ephi=\nabla \phi$ and $\ez=\nabla Z$ with $R$, $\phi$, $Z$ functions of ($X$,$Y$,$Z$). By definition of the basis we have $\er\times\ephi=-\ez$, $\ephi\times\ez=-\er$ and  $\ez\times\er=-\ephi$.\\
\begin{figure}[t]
\begin{center}
    \includegraphics[scale=.35]{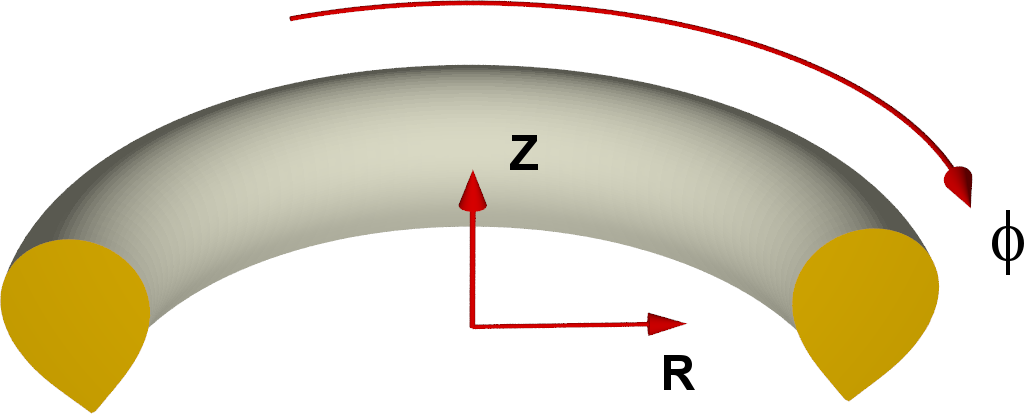}
        \caption{Illustration of the cylindrical coordinate system used in JOREK.}
\end{center}
\end{figure}

The domain is defined by $(R,Z,\phi)\in\Omega=D\times [0,2\pi[$. 
To finish we define the different differential operators used for the calculation:
$$
\nabla f=\dr (f)\er+\frac{1}{R}\dphi (f)\ephi+ \dz (f)\ez
$$
$$
\nabla_{pol} f=\dr (f)\er+ \dz (f)\ez
$$
$$
\nabla\cdot \mathbf{f}=\frac{1}{R}\dr (R f_R)+\frac{1}{R}\dphi (f_{\phi})+ \dz (f_Z)
$$
$$
\nabla\times \mathbf{f}=\left(\frac{1}{R}\dphi f_Z-\dz f_{\phi}\right)\er+\left(\dz f_R-\dr f_{Z}\right)\ephi+\frac{1}{R}\left(\dr(R f_{\phi})-\dphi f_{R}\right)\ez
$$
$$
\triangle^{*}f=R^2\nabla\cdot\left(\frac{1}{R^2}\nabla^{\perp}f\right)=R\dr \left(\frac{1}{R}\dr f\right)+\partial_{ZZ} f
$$
$$
\triangle_{pol}=\nabla \cdot (\nabla_{pol} f)=\frac{1}{R}\dr(R\dr f)+\partial_{ZZ} f
$$
$$
[a,b]=\ephi\cdot(\nabla a \times \nabla b)=\dr a\dz b-\dz a \dr b.
$$
The variables associated to the reduced MHD models are the poloidal magnetic flux $\psi$, the electrical potential $u$,  the density $\rho$, the temperature $T$  and the parallel velocity $v_{||}$. We introduce two additional variables: the toroidal current $j$ defined by $j=\gs \psi$ and the vorticity $w$ defined by $w=\lp u$. This procedure is used to break some high order operators into lower order ones.
For the integration we denote by $dW= R dR dZ$ the cylindrical measure and $dV=dR dZ$. When no measure is given explicitly, $dR dZ$ is used.

\subsection{Derivation of the model}
The starting point of our derivation is the the following resistive MHD model
\begin{equation}\label{MHDre1}
\left\{\begin{array}{l}
\ds \dt \rho+\nabla\cdot(\rho \vv)=0
\\
\ds\rho \dt \vv+\rho\vv\cdot\nabla \vv+\nabla p=\vJ\times\vB
\\
\ds \dt p+ \vv\cdot\nabla p+\gamma p \nabla \cdot\vv=0
\\
\ds\dt \vB=\rot (\vv\times\vB-\eta\vJ)\\
\end{array}\right.
\end{equation}
We do not treat the viscosity term in the following derivation, but discuss it briefly at the end of Section \ref{sec:final_model}.

\subsubsection{Magnetic poloidal flux equation}
We use the magnetic field $\vB=\vB_{\phi}+\vB_{pol}$ given by (\ref{BBB}). Since $\nabla \cdot \vB=0$ and $\dt F_0=0$ we get $\dt \vB=\dt\left(\frac{F_0}{R}\ephi+\nabla \times\left(\frac{1}{R}\psi\ephi\right)\right)=\nabla \times\left(\dt(\frac{1}{R}\psi\ephi)\right)$. Consequently the equation on the magnetic field in (\ref{MHDre1}) becomes
\begin{equation}\label{In1}
\ds\nabla \times\left(\dt(\frac{1}{R}\psi\ephi)\right)=\rot (\vv\times\vB-\eta\vJ)
\end{equation}
with $\vJ=\rot \vB$.
The equation becomes
\begin{equation}\label{In2}
\ds\dt\left(\frac{1}{R}\psi\ephi\right)=\vv\times\vB-\eta\vJ+\nabla V
\end{equation}
with $V$ a potential.
We begin by estimating the term $\vv\times\vB$. Since $\vB\times\vB=0$ we obtain $\vv\times\vB=\left(-R^2\nabla \times\left(\frac{1}{R}u\ephi\right)\right)\times\left(\frac{F_0}{R}\ephi+\nabla \times\left(\frac{1}{R}\psi\ephi\right)\right)$ which gives
\begin{align*}
\vv\times\vB & =\vv_{pol}\times\vB=\left(-R\dz u\er+R\dr u\ez\right)\times \left(\frac{F_0}{R}\ephi+\frac{1}{R}\dz\psi\er-\frac{1}{R}\dr\psi\ez\right)\\
& =F_0\left(\dz u\ez+\dr u\er\right)+[\psi,u]\ephi.
\end{align*}
Now we study the term $\vJ=\rot \vB$.
$$
\vJ=\rot(F_0\nabla \phi)+\rot(\nabla \psi\times \nabla \phi).
$$
Since $F_0$ is constant, using the properties of curl and gradient operators we have $\rot\nabla\phi=0$.  So
$$
\vJ=\rot(\nabla \psi\times \nabla \phi)=\rot\left(\frac{1}{R}\dz\psi\er-\frac{1}{R}\dr \psi \ez\right).
$$
Since $\rot\er=\rot \nabla R=0$ and $\rot\ez=\rot \nabla Z=0$ we have
$$
\vJ=\rot(\nabla \psi\times \nabla \phi)=\nabla\left(\frac{1}{R}\dz\psi\right)\times\er-\nabla\left(\frac{1}{R}\dr\psi\right)\times\ez.
$$
Therefore expanding the gradient for each component we obtain 
$$
\vJ=\rot(\nabla \psi\times \nabla \phi)=-\frac{1}{R}\partial_{ZZ}\psi\ephi+\frac{1}{R^2}\dphi(\dz\psi)\ez-\dr(\frac{1}{R}\dr\psi)\ephi+\frac{1}{R^2}\dphi(\dr\psi)\er,
$$
and using the definition of the Grad-Shafranov diffusion operator we have
$$
\rot(\nabla \psi\times \nabla \phi)=-\frac{1}{R}\triangle^{*}\psi\ephi+\frac{1}{R^2}\dphi(\dz\psi)\ez+\frac{1}{R^2}\dphi(\dr\psi)\er.
$$
We plug together all the terms to obtain
\begin{align*}
\ds\dt\left(\frac{1}{R}\psi\ephi\right) & = + F_0\left(\dz u\ez+\dr u\er\right)+[\psi,u]\ephi\\
&\quad -\eta\left[-\frac{1}{R}\triangle^{*}\psi\ephi+\frac{1}{R^2}\dphi(\dz\psi)\ez+\frac{1}{R^2}\dphi(\dr\psi)\er\right]+\nabla V.
\end{align*}
Now we multiply the previous equation by $\er$ and after by $\ez$ to obtain the expressions of the $R$ and $Z$ derivatives of V:
\begin{equation}\label{In5}
\left\{\begin{array}{l}
\ds \dr V =-F_0\dr u+\frac{\eta}{R^2}\partial_{\phi R}\psi,\\
\\
\ds \dz V =-F_0\dz u+\frac{\eta}{R^2}\partial_{\phi Z}\psi.
\end{array}\right.
\end{equation} 
 Consequently $ V=-F_0 u+\frac{\eta \partial_{\phi} \psi}{R^2}+C$ and $\nabla V\cdot\ephi =-\frac{F_0}{R}\dphi u+\eta \frac{\partial_{\phi\phi} \psi}{R^3}$. We have assumed that the constant $C$ is equal to zero to obtain the electrical potential that is usually chosen and compatible with the choice of the poloidal velocity.
 
 This definition of $V$ gives the final result.
\begin{equation}\label{psifinal}
 \ds\dt\psi=R[\psi,u]+\eta\triangle^{*}\psi-F_0\dphi u+\eta \frac{\partial_{\phi\phi} \psi}{R^2}
\end{equation}
with $j=-R\vJ\cdot\ephi=\triangle^{*}\psi$ the toroidal current. 

\subsubsection{Poloidal momentum equation}
To obtain an equation on the electric potential we apply the projection operator $\ephi \cdot\nabla\times\left(R^2.....\right)$ in the poloidal plane to the momentum equation. The equation obtained is 
\begin{equation}\label{u1}
\ephi \cdot\nabla\times\left[R^2 \left(\rho\dt\vv=-\rho\vv\cdot\nabla \vv-\nabla p+\vJ\times\vB\right)\right].
\end{equation}
We begin by considering the first term of (\ref{u1}): $\ephi\cdot\nabla\times( R^2\rho \dt \mathbf{v})$. Using the definition of $\vv_{pol}$ we obtain
$$
\ds\nabla\times(R^2\rho\dt \mathbf{v}_{pol})=\nabla\times \left[ -\rho R^3\dz(\dt u)\er+\rho R^3\dr(\dt u)\ez\right],
$$ 
and
$$
\ds\ephi\cdot\nabla\times( R^2\rho\dt \mathbf{v}_{pol})=\dz\left(\rho R^3 \dz(\dt u)\right)+\dr\left(\rho R^3\dr(\dt u)\right).
$$ 
By definition of the gradient and the divergence we obtain 
\begin{equation}\label{u2}
\ephi\cdot\nabla\times( R^2\rho \dt \mathbf{v}_{pol})=R\nabla\cdot(\rho R^2\nabla_{pol}\dt u).
\end{equation}
We consider the term associated to the time derivative of the parallel velocity $\ephi \cdot\rot (R^2 \rho\dt \vv_{||})=\ephi \cdot\rot \left[R^2\rho\dt (v_{||}\vB)\right]$. Developing $\vB$ we obtain
$$
\rot (R^2 \rho\dt \vv_{||})=\rot \left[R^2 \rho \left[\dt \left(v_{||}\frac{\dz \psi}{R}\right)\er -\dt \left(v_{||}\frac{\dr \psi}{R}\right)\ez+ \frac{F_0}{R}\dt v_{||}\ephi\right]\right].
$$
After some algebra
\begin{equation}\label{u3}
\ephi \cdot\rot (R^2 \rho\dt \vv_{||})=-\dz \left[R\rho \dt(v_{||}\dz \psi)\right]-\dr \left[R\rho \dt(v_{||}\dr \psi)\right]=
-R \nabla \cdot \left[\rho \dt (v_{||}\nabla_{pol} \psi)\right].
\end{equation}
Secondly, we study the current term $\ephi\cdot\nabla\times( R^2(\vJ\times\vB))$.  We recall the form of the current 
\begin{equation}\label{u4}
\vJ=\nabla\times\vB=\frac{1}{R^2}\left(\partial_{R\phi}\psi \er-R j\ephi+\partial_{Z\phi}\psi\ez
\right)
\end{equation}
computed in the previous subsection. So, using (\ref{u3}) - (\ref{u4}) we have
\begin{align*}
R^2\vJ\times\vB  =& \frac{1}{R}\left[(-R j \dr \psi+F_0\partial_{Z\phi}\psi)\er+(-R j \dz \psi-F_0\partial_{R\phi}\psi)\ez
\right]\\
 & +\frac{1}{R}\left[-((\dr\psi)(\partial_{R\phi}\psi)+(\dz\psi)(\partial_{Z\phi}\psi))\ephi\right].
\end{align*}
Applying the operator $\ephi\cdot\nabla\times(R^2 ...)$ we obtain
$$
\ephi\cdot\nabla\times(R^2\vJ\times\vB)=\dz\left(j\dr\psi\right)-\dz\left(\frac{F_0}{R}\partial_{Z\phi}\psi\right)-\dr\left(j\dz\psi\right)-\dr\left(\frac{F_0}{R}\partial_{R\phi}\psi\right).
$$
A final calculation gives the following result
\begin{equation}\label{u5}
\ephi\cdot\nabla\times(R^2\vJ\times\vB)=[\psi,j]-\frac{F_0}{R}\dphi \gs\psi=[\psi,j]-\frac{F_0}{R}\dphi j.
\end{equation}
For the pressure term, trivial computations allow to obtain the following result
\begin{equation}\label{u6}
\rot (R^2\nabla p)\cdot\ephi=-\dr(R^2)\dz p=-[p,R^2]=[R^2,p].
\end{equation}
The last term considered is  $\ephi\cdot\nabla\times(R^2\rho \vv\cdot\nabla\vv)$. Firstly we study the part which depends only on the poloidal velocity: $\ephi\cdot\nabla\times(R^2\rho \vv_{pol}\cdot\nabla\vv_{pol})$. To begin we denote by $\alpha=-R \dz u $,  $\beta=R\dr u $ and $\widehat{\rho}=R^2\rho$. So
$$
\vv_{pol}= \alpha \er+\beta \ez\mbox{ and } \vv_{pol} \cdot\nabla \vv_{pol} =(\alpha \dr \alpha+\beta \dz \alpha)\er+(\alpha \dr \beta+\beta \dz \beta)\ez.
$$
To estimate this term we propose the following decomposition
\begin{equation}\label{u7}
\rot (\widehat{\rho} \vv_{pol} \cdot\nabla \vv_{pol})\cdot\ephi=(\widehat{\rho} \rot ( \vv_{pol} \cdot\nabla \vv_{pol})+ \nabla \widehat{\rho}\times \vv_{pol} \cdot\nabla \vv_{pol})\cdot\ephi=A+B.
\end{equation}
One has the identities
$$
A= -\widehat{\rho}\left[\dz(\alpha \dr \alpha+\beta \dz \alpha)-\dr(\alpha \dr \beta+\beta \dz \beta)\right],
$$
and
$$
A= -\widehat{\rho}\left[\alpha \dr (\dz \alpha-\dr \beta)+\beta \dz (\dz \alpha-\dr \beta)+(\dr \alpha+\dz \beta)(\dz \alpha-\dr \beta)\right].
$$
Using $(\dz \alpha-\dr \beta)=-R\lp u=-R w$ and $\widehat{\rho}(\dr \alpha+\dz \beta)=-\widehat{\rho}[R,u]$ we obtain
$A =-R \widehat{\rho}[R w,u]-R \widehat{\rho} w [R,u]$.
In a first time we estimate the term $B$. The definition of the vector product gives
$$
B=-\left[\dz(\widehat{\rho})(\alpha \dr \alpha+\beta \dz \alpha)-\dr(\widehat{\rho})(\alpha \dr \beta+\beta \dz \beta)\right]
$$
which we can rewrite in the following form
$$
B=-\left(\dz(\widehat{\rho})\left[\dr(\frac12\alpha^2+\frac12\beta^2)+\beta(\dz \alpha-\dr \beta)\right]-\dr(\widehat{\rho})\left[\dz(\frac12\alpha^2+\frac12\beta^2)-\alpha(\dz \alpha-\dr \beta)\right]\right).
$$
Using $\dz \alpha-\dr \beta=-R w$, we obtain the final expression of the term $B$, which reads
$$
B=-\frac12[R^2|\nabla_{pol}u|^2,\widehat{\rho}]-\dz(\widehat{\rho})\beta(\dz \alpha-\dr \beta)-\dr(\widehat{\rho})\alpha(\dz \alpha-\dr \beta)=-\frac12[R^2|\nabla_{pol}u|^2,\widehat{\rho}]-R^2 w [\widehat{\rho},u].
$$
To finish we sum up $A$ and $B$ to obtain
$$
\rot (\widehat{\rho} \vv_{pol} \cdot\nabla \vv_{pol})\ephi=-\frac12[R^2|\nabla_{pol}u|^2,\widehat{\rho}]-R[R\widehat{\rho} w,u]-R\widehat{\rho} w [R,u].
$$
Therefore
\begin{equation}\label{u8}
\rot (\widehat{\rho} \vv_{pol} \cdot\nabla \vv_{pol})\ephi=-\frac12[R^2|\nabla_{pol}u|^2,\widehat{\rho}]-[R^2\widehat{\rho} w,u].
\end{equation}
~\\
 At this moment of the derivation using the equation on the velocity in the resistive MHD (\ref{MHDre1}), using the projection $\ephi\cdot\nabla\times( R^2\rho \dt \mathbf{v})$ and neglecting all the terms in the velocity equation which depend on the parallel velocity we have obtained the equation on $u$ implemented in the code. Now we propose to derive the terms neglected in the code which correspond to the following cross terms between the parallel and poloidal velocities and given by
 \begin{equation}\label{u9}
\ephi\cdot\rot R^2\rho\left(\dt \mathbf{v}_{||}+\mathbf{v}_{||}\cdot\nabla \mathbf{v}_{||}+\mathbf{v}_{||}\cdot\nabla \mathbf{v}_{pol}+\mathbf{v}_{pol}\cdot\nabla \mathbf{v}_{||}\right).
\end{equation}
Firstly we consider the term $\ephi \cdot\rot (R^2\rho \vv_{||}\cdot\nabla \vv_{||})$. We begin by splitting the term into two parts 
$$
A=\widehat{\rho}\nabla \times (v_{||}\vB\cdot\nabla (v_{||}\vB))\cdot \ephi \mbox{ and }  B=\nabla \widehat{\rho}\times (v_{||}\vB\cdot\nabla (v_{||}\vB)) \cdot \ephi
$$
and we define $v_{||}\vB =v_{||}a\er+v_{||}b\ez+v_{||}c\ephi$ with $a=\frac{\dz \psi}{R}$, $b=-\frac{\dr \psi}{R}$ and $c=\frac{F_0}{R}$, consequently
\begin{align*}
\vB\cdot\nabla (v_{||}\vB)=&+\left[a\dr (a v_{||})+b \dz (a v_{||})+\frac{c}{R}\dphi (a v_{||})\right]\er+\left[ a\dr (bv_{||})+b\dz (bv_{||})+\frac{c}{R}\dphi (bv_{||})\right]\ez\\
&+\left[a\dr (cv_{||})+b\dz (cv_{||})+\frac{c}{R}\dphi (cv_{||})\right]\ephi-\left(v_{||}^2 \frac{c^2}{R}\right)\er+\left(v_{||}^2 \frac{c}{R} a\right)\ephi.
\end{align*}
The term $B$ can be decomposed as $B =C+D$ with
$$
C=(\dr \widehat{\rho})\left[av_{||}\dr (bv_{||})+bv_{||}\dz (bv_{||})\right]-(\dz \widehat{\rho})\left[av_{||}\dr (av_{||})+bv_{||}\dz (av_{||})\right],
$$
and
\begin{align*}
D & = +(\dr \widehat{\rho})\left[\frac{cv_{||}}{R}\dphi (bv_{||})\right]-(\dz \widehat{\rho})\left[\frac{cv_{||}}{R}\dphi (av_{||})\right]+\dz(\widehat{\rho})(v_{||}^2\frac{c^2}{R})\\
& = -\frac{F_0}{R^3}v_{||}\left(\nabla_{pol} \widehat{\rho}\cdot\dphi(v_{||}\nabla_{pol} \psi)\right)+\dz(\widehat{\rho})(v_{||}^2\frac{F_0^2}{R^3}).
\end{align*}
We rewrite the term $C$ to obtain
\begin{align*}
C= &\dr(\widehat{\rho})\left[\dz(\frac12 v_{||}^2a^2+\frac12 v_{||}^2b^2)+av_{||}\dr (bv_{||})-av_{||}\dz (av_{||})\right]\\
& -\dz(\widehat{\rho})\left[\dr(\frac12 v_{||}^2a^2+\frac12 v_{||}^2b^2)+bv_{||}\dz (av_{||})-bv_{||}\dr (bv_{||})\right]
\end{align*}
which is equal to
\begin{equation}\label{u11}
C=\frac12[\widehat{\rho},v_{||}^2\frac{|\nabla_{pol} \psi|^2}{R^2}]-v_{||}\left[\dr(\widehat{\rho})a(\dz (av_{||})-\dr (bv_{||}))+\dz(\widehat{\rho})b(\dz (av_{||})-\dr (bv_{||}))\right].
\end{equation}
We remark that $\dz (av_{||})-\dr (bv_{||})=\frac{v_{||}}{R}\gs \psi+\frac{1}{R}(\gradp v_{||}\cdot \gradp \psi)$. Using this result we obtain the following expression for $B$:\\
~\\
\begin{align*}
B = & +\frac12[\widehat{\rho},v_{||}^2|\vB_{pol}|^2]-\frac{v_{||}^2}{R^2} j [\widehat{\rho},\psi]-\frac{v_{||}}{R^2}(\gradp v_{||}\cdot \gradp \psi)[\widehat{\rho},\psi]\\
& -\frac{F_0}{R^3}v_{||}(\gradp\widehat{\rho}\cdot\dphi(v_{||}\gradp \psi))+\dz(\widehat{\rho})(v_{||}^2\frac{F_0^2}{R^3}),
\end{align*}
with $\vB_{pol}=\frac{1}{R}\nabla \psi \times \ephi$.\\ \\
Now we study the term $A=\widehat{\rho}\nabla \times (v_{||}\vB\cdot\nabla (v_{||}\vB))\cdot \ephi$ which is equal to
\begin{align*}
A = & - \widehat{\rho}\dz\left[ a v_{||}\dr (a v_{||})+b v_{||}\dz (av_{||})+\frac{c v_{||}}{R}\dphi (a v_{||})-v_{||}^2\frac{c^2}{R}\right]\\
&+\widehat{\rho}\dr\left[a v_{||}\dr (b v_{||})+b v_{||}\dz (b v_{||}) +\frac{c v_{||}}{R}\dphi (b v_{||})\right].
\end{align*}
We split these terms into two terms $A=A_1+A_2$ defined by $A_1=-\widehat{\rho}\dz(av_{||}\dr (av_{||})+bv_{||}\dz (av_{||}))+\widehat{\rho}\dr(av_{||}\dr (bv_{||})+bv_{||}\dz (bv_{||}))$
and $A_2=-\widehat{\rho}\dz(\frac{cv_{||}}{R}\dphi (av_{||}))+\widehat{\rho}\dr(\frac{cv_{||}}{R}\dphi (bv_{||}))+\widehat{\rho}\dz (v_{||}^2\frac{c^2}{R})$.\\
~\\
Factorizing  the term $A_1$ we obtain
\begin{align*}
A_1 = & -\widehat{\rho}\left[av_{||}\dr(\dz (av_{||})-\dr (bv_{||}))+bv_{||}\dz(\dz (av_{||})-\dr (bv_{||}))\right]\\
 &  -\widehat{\rho}\left[(\dr (a v_{||})+\dz (bv_{||}) )(\dz (a v_{||})-\dr (bv_{||}))\right].
\end{align*}
~\\
Using that $\dz (a v_{||})-\dr (bv_{||})=\frac{v_{||}}{R}\gs \psi+\frac{1}{R}(\gradp v_{||}\cdot\gradp \psi)$ and $\dr (av_{||})+\dz (bv_{||})=[\frac{v_{||}}{R},\psi]$ we obtain
\begin{align*}
A_1 = & -\frac{\widehat{\rho}v_{||}}{R}[\frac{v_{||}}{R} j,\psi]-\frac{\widehat{\rho}v_{||}}{R}[\frac{1}{R}(\gradp v_{||}\cdot\gradp \psi),\psi]\\
& -\frac{\widehat{\rho}v_{||}}{R} j [\frac{v_{||}}{R},\psi]-\frac{\widehat{\rho}}{R}(\gradp v_{||}\cdot\gradp \psi)[\frac{v_{||}}{R},\psi].
\end{align*}
The properties of the Poisson bracket allow to conclude 
\begin{equation}\label{u12}
A_1=-\widehat{\rho}[\frac{v_{||}^2}{R^2} j ,\psi]-\widehat{\rho}[\frac{v_{||}}{R^2}(\gradp v_{||}\cdot\gradp \psi),\psi].
\end{equation}
For the term $A_2$, some computations allow to obtain the result
\begin{equation}\label{u13}
A_2=-\widehat{\rho}\dz\left(\frac{F_0}{R^3}v_{||}\dphi(v_{||}\dz \psi)\right)-\widehat{\rho}\dr\left(\frac{F_0}{R^3}v_{||}\dphi(v_{||}\dr \psi)\right)+\widehat{\rho}\dz\left(\frac{F_0^2}{R^3}v_{||}^2\right).
\end{equation}
At the end using the properties of the  Poisson bracket  and the product of derivatives  we obtain
\begin{align*}
\ephi \cdot\rot (R^2\rho \vv_{||}\cdot\nabla \vv_{||})= & -[\rho v_{||}^2 j,\psi]-[\rho v_{||}(\gradp v_{||}\cdot\gradp \psi),\psi]+\frac12[\widehat{\rho},v_{||}^2|\vB_{pol}|^2]\\
& -\dz\left(\widehat{\rho}\frac{F_0}{R^3}v_{||}\dphi(v_{||}\dz \psi)\right)-\dr\left(\widehat{\rho}\frac{F_0}{R^3}v_{||}\dphi(v_{||}\dr \psi)\right)\\
& +\dz\left(\widehat{\rho}\frac{F_0^2}{R^3}v_{||}^2\right) \numberthis \label{u12b}.
\end{align*}
To finish the derivation associated with poloidal velocity, we study the last term $\ephi \cdot\rot (\widehat{\rho} \vv_{pol}\cdot\nabla \vv_{||}+\widehat{\rho} \vv_{||}\cdot\nabla \vv_{pol})$. Firstly we note
$$
v_{||}\vB\cdot\nabla \vv_{pol}=v_{||}\left[a\dr \alpha + b\dz \alpha +\frac{c}{R}\dphi \alpha\right]\er+v_{||}\left[a\dr \beta + b\dz \beta +\frac{c}{R}\dphi \beta\right]\ez+v_{||}\frac{c}{R}\alpha \ephi
$$
and
$$
\vv_{pol}\cdot\nabla (v_{||}\vB)=\left[\alpha\dr(a v_{||})+\beta\dz(a v_{||})\right]\er+\left[\alpha\dr(b v_{||})+\beta\dz(b v_{||})\right]\ez+
\left[\alpha\dr(c v_{||})+\beta\dz(c v_{||})\right]\ephi.
$$
The term $\ephi \cdot\rot (\widehat{\rho} \vv_{pol}\cdot\nabla (v_{||}\vB)+ \widehat{\rho}v_{||}\vB\cdot\nabla \vv_{pol})$ can be split into two terms
\begin{align*}
(A) &=\nabla \widehat{\rho} \times ( \vv_{pol}\cdot\nabla (v_{||}\vB)+ v_{||}\vB\cdot\nabla \vv_{pol})\cdot\ephi\\
(B) &=\widehat{\rho}\rot( \vv_{pol}\cdot\nabla (v_{||}\vB)+ v_{||}\vB\cdot\nabla \vv_{pol})\cdot\ephi .
\end{align*}
Using our notation we obtain that $(A)=(A1)+(A2)$ with
\begin{align*}
(A1) = & - \dz \widehat{\rho} \left[v_{||}a\dr \alpha+\alpha \dr (a v_{||})+v_{||}b\dz \alpha+\beta\dz (a v_{||})\right]\\
&  +\dr \widehat{\rho} \left[v_{||}b\dz \beta+\beta \dz (b v_{||})+v_{||}a\dr \beta+\alpha\dr (b v_{||})\right],
\end{align*}
and
$$
(A2)=-\dz \widehat{\rho} \frac{c v_{||}}{R}\dphi\alpha+\dr \widehat{\rho} \frac{c v_{||}}{R}\dphi\beta.
$$
Straightforward calculations show that the term $(A2)$ is equal to $(A2)=\frac{F_0}{R}v_{||}(\nabla_{pol} \widehat{\rho} \cdot \nabla_{pol} (\dphi u))$. Now we consider the term $(A1)$ which can me rewritten in the following form
\begin{align*}
(A1)  = &-\dz (\widehat{\rho}) \left[ \dr(a v_{||}\alpha+b v_{||}\beta) -v_{||}b \dr \beta-\beta\dr (b v_{||})+v_{||}b\dz\alpha+\beta\dz(a v_{||})\right]\\
& +\dr (\widehat{\rho}) \left[ \dz(a v_{||}\alpha+b v_{||}\beta) -v_{||}a \dz \alpha-\alpha\dz (a v_{||})+\alpha \dr(b v_{||})+v_{||}a\dr\beta\right].
\end{align*}
Using the definition of the different coefficients we obtain that $(A1)=-[\widehat{\rho},v_{||}(\gradp \psi\cdot\gradp u)]+(A3)+(A4)$ with
\begin{align*}
(A3) &=-\dz (\widehat{\rho}) v_{||} (-b\dr \beta +b\dz \alpha)+\dr (\widehat{\rho}) v_{||} (-a\dz \alpha +a\dr \beta)\\
(A4) &=-\dz (\widehat{\rho}) (-\beta\dr (v_{||} b) +\beta\dz (a v_{||}))+\dr (\widehat{\rho} ) (-\alpha\dz (v_{||} a) +\alpha\dr (b v_{||})).
\end{align*}
Now we consider the term $(A3)$ which can be factorized in the following way 
$$
(A3)=(\dz (\widehat{\rho}) v_{||} b+\dr (\widehat{\rho}) v_{||}a)(-\dz \alpha +\dr \beta).
$$
Using that $(\dz \alpha -\dr \beta)=-R\lp u$ we obtain that $(A3)=v_{||} w [\widehat{\rho},\psi]$ and $(A4)=-(\dz \widehat{\rho} \beta + \dr \widehat{\rho} \alpha)(\dz (a v_{||})-\dr (b v_{||}))$. \\
We known that $\dz (a v_{||})-\dr (b v_{||})=\frac{1}{R}(\gradp \psi\cdot\gradp v_{||})+\frac{1}{R}v_{||} j $, consequently at the end we have $(A4)=[u,\widehat{\rho}](\gradp \psi\cdot\gradp v_{||}+v_{||} j)$ with $(\dz (\widehat{\rho}) \beta + \dr (\widehat{\rho}) \alpha)=R[u,\widehat{\rho}]$. Putting all the terms together we obtain that 
\begin{align*}
(A) & =-[u,\widehat{\rho}](\gradp \psi\cdot\gradp v_{||})-[u,\widehat{\rho}]v_{||} j -[\widehat{\rho},v_{||}(\gradp \psi\cdot\gradp u)]\\
&\quad+v_{||} w [\widehat{\rho},u]+\frac{F_0}{R}v_{||}(\gradp \widehat{\rho}\cdot\gradp (\dphi u)).
\end{align*}
Now we consider the term $(B)=\widehat{\rho}\nabla \times (\vv_{pol}\cdot\nabla v_{||}\vB+ v_{||}\vB\cdot\nabla \vv_{pol})\cdot\ephi$ decomposed into two terms $(B)=(B1)+(B2)$ with
\begin{align*}
(B1) &=-\widehat{\rho}\dz\left[v_{||}a\dr \alpha+v_{||}b\dz \alpha+\alpha \dr(a v_{||})+\beta \dz(a v_{||})\right]\\
&\quad +\widehat{\rho}\dr \left[v_{||}a\dr \beta+v_{||}b\dz \beta+\alpha \dr(b v_{||})+\beta \dz(b v_{||}) \right]
\end{align*}
and
$$
(B2)=\widehat{\rho}\left(-\dz(\frac{c v_{||}}{R}\dphi\alpha)+\dr(\frac{c v_{||}}{R}\dphi\beta)\right).
$$
We consider the term $(B1)$. We begin by expending $(B1)$ and after rearranging terms we obtain
\begin{align*}
(B1) = &-\widehat{\rho}\left[v_{||}a\dr(\dz \alpha-\dr \beta)+v_{||}b\dz(\dz \alpha-\dr \beta)\right.\\
& \left.+\alpha\dr(\dz (a v_{||})-\dr (b v_{||}))+\beta\dz(\dz (a v_{||})-\dr (b v_{||}))\right.\\
& \left.+(\dr(v_{||}a)+\dz(v_{||}b))(\dz \alpha-\dr \beta)+(\dr\alpha+\dz\beta)(\dz (a v_{||})-\dr (b v_{||}))\right].
\end{align*}
We use that $(\dz \alpha-\dr \beta)=-R\lp u$ and $\dz (a v_{||})-\dr (b v_{||})=\frac{v_{||}}{R} j+\frac{1}{R}(\gradp \psi\cdot\gradp v_{||})$. Using $\dr(v_{||}a)+\dz(v_{||}b)=[\frac{v_{||}}{R},\psi]$ and $(\dr\alpha+\dz\beta)=[u,R]$, we write the term $(B1)$ in the following form
\begin{align*}
(B1) =& +\frac{1}{R}\widehat{\rho}v_{||}[R w,\psi]+\widehat{\rho}R w [\frac{v_{||}}{R},\psi]-R\widehat{\rho}[u,\frac{v_{||}}{R} j ]-\widehat{\rho}[u,R]\frac{ v_{||}}{R} j\\
= & -R\widehat{\rho}[u,\frac{1}{R}(\gradp \psi\cdot\gradp v_{||})] -\frac{\widehat{\rho}}{R}[u,R](\gradp \psi\cdot\gradp v_{||})\\
= & + \widehat{\rho}[v_{||} w,\psi]-\widehat{\rho}[u,v_{||} j ]-\widehat{\rho}[u,(\gradp \psi\cdot\gradp v_{||})].
\end{align*}
The term $(B2)$ is equal to $\widehat{\rho} R\nabla\cdot(\frac{F_0}{R^2}v_{||}\gradp (\dphi u))$, consequently we obtain
\begin{equation}\label{u15}
(B)=\widehat{\rho}[v_{||} w,\psi]-\widehat{\rho}[u,v_{||} j]-\widehat{\rho}[u,(\gradp\psi\cdot\gradp v_{||})]+\widehat{\rho} R\nabla.(\frac{F_0}{R^2}v_{||}\gradp(\dphi u)).
\end{equation}
All together we have derived the following term
\begin{align*}
\ephi \cdot\rot (\widehat{\rho} \vv_{pol}\cdot\nabla \vv_{||}+\widehat{\rho} \vv_{||}\cot\nabla \vv_{pol})= &-[\widehat{\rho},v_{||}(\gradp \psi\cdot\gradp u)]+[\widehat{\rho}v_{||} w,\psi]-[u,\widehat{\rho}v_{||} j ]\\
&-[u,\widehat{\rho}(\gradp \psi\cdot\gradp v_{||})]+ R\nabla\cdot\left(\widehat{\rho}\frac{F_0}{R^2}v_{||}\gradp (\dphi u)\right). \numberthis \label{u13b}
\end{align*}

\subsubsection{Equation on $\rho$ and $T$}
For the thermodynamic equations $\dt \rho =-\rho \nabla \cdot \vv-\vv \cdot\nabla \rho$ and $\dt p =-\gamma p \nabla \cdot \vv-\vv \cdot\nabla p$ , we propose to rewrite the equations in order to obtain a dependency on $u$ and $v_{||}$. 
We begin with
\begin{align*}
\vv_{pol}\cdot\nabla \rho = & +(-R\nabla u \times \ephi)\cdot\nabla \rho \\
= & + (-R\dz u \er +R \dr u \ez)\cdot\left[\dr \rho \er +\frac{1}{R}\dphi \rho \ephi +\dz \rho \ez\right]\\
= & -R(\dz u)(\dr \rho)+R(\dr u)(\dz \rho)=-R[\rho,u].
\end{align*}
Then we compute the second term $\rho \nabla \cdot \vv_{pol}$:
$$
\rho \nabla \cdot \vv_{pol}=-\rho\frac{1}{R}\dr(R^2\dz u)+\rho \dz (R \dr u)=-2\rho \dz u.
$$
Now we derive the term associated to the parallel velocity $\vv_{||}=v_{||}\vB$:
\begin{align*}
v_{||}\vB\cdot\nabla \rho &=v_{||}\left[\frac{F_0}{R}\ephi+\frac{1}{R}\dz \psi \er -\frac{1}{R}\dr \psi \ez\right]\cdot\nabla \rho\\
 &=v_{||}\frac{F_0}{R^2}\dphi \rho + \frac{1}{R}v_{||}[\rho,\psi].
\end{align*}
The second term is
\begin{align*}
\rho \nabla \cdot (v_{||}\vB) &=\rho \nabla \cdot \left[v_{||}(\frac{F_0}{R}\ephi+\frac{1}{R}\dz \psi \er -\frac{1}{R}\dr \psi \ez)\right]\\
 & =\frac{ \rho}{R}[v_{||},\psi]+\frac{\rho F_0}{R^2}\dphi v_{||}.
\end{align*}
Consequently we obtain
\begin{equation}\label{rho}
\dt \rho = R[\rho,u]+2\rho \dz u-\frac{v_{||} F_0}{R^2}\dphi \rho-\frac{v_{||}}{R}[\rho,\psi]-\frac{ \rho}{R}[v_{||},\psi]-\frac{\rho F_0}{R^2}\dphi v_{||}
\end{equation}
and
\begin{equation}\label{T}
\dt p = R[p,u]+2\gamma p \dz u-\frac{v_{||} F_0}{R^2}\dphi p-\frac{v_{||}}{R}[p,\psi]-\frac{\gamma p}{R}[v_{||},\psi]-\frac{\gamma p F_0}{R^2}\dphi v_{||}.
\end{equation}

\subsubsection{Equation on the parallel velocity}
We consider the equation $\rho\dt \vv=-\rho \vv\cdot \nabla \vv -\nabla p+\mathbf{J}\times  \vB$. To obtain the equation on $v_{||}$ we project the momentum equation applying the operator $\vB\cdot (...)$. Firstly we remark  that $\vB \cdot (\mathbf{J}\times  \vB)=Det(\vB,\mathbf{J},\vB)=0$. Secondly we consider $ \vB\cdot \rho\dt (\vv_{||}+\vv_{pol})$.
Using the definition of $\vB$ we prove that the term $\vB\cdot \rho\dt (v_{||}\vB$) is equal to 
\begin{equation}\label{dtvpar1}
\rho|\vB|^2\dt v_{||}+\rho v_{||}\frac{1}{R^2}\gradp \psi\cdot \gradp (\dt \psi).
\end{equation}
For the poloidal term $\vB \cdot \rho(\dt \vv_{pol})$, straightforward computations show that this term is given by 
\begin{equation}\label{dtvpar2}
\vB \cdot \rho(\dt \vv_{pol})=-\rho\nabla_{pol} \psi\cdot \nabla_{pol} (\dt u).
\end{equation}
For the pressure term $\vB\cdot\nabla p$, we obtain
\begin{equation}\label{vp1}
\vB\cdot\nabla p=(\frac{F_0}{R}\ephi+\frac{1}{R}\dz \psi \er-\frac{1}{R}\dr \psi \ez)(\dr p \er+\frac{1}{R}\dphi p \ephi+\dz p \ez)=\frac{F_0}{R^2}\dphi p+\frac{1}{R}[p,\psi].
\end{equation}
Now we consider the following terms (the four last terms which are neglected in the model implemented in the code JOREK):
\begin{equation}\label{vp2}
\vB\cdot \rho\left(\vv_{||}\cdot\nabla \vv_{||}+\mathbf{v}_{||}\cdot\nabla \mathbf{v}_{pol}+\mathbf{v}_{pol}\cdot\nabla \mathbf{v}_{||}+\mathbf{v}_{pol}\cdot\nabla \mathbf{v}_{pol}\right).
\end{equation}
Firstly we study $\vB \cdot(\rho\vv_{||}\cdot\nabla \vv_{||})=\vB \cdot (\rho v_{||}\vB \cdot\nabla( v_{||}\vB))$. For this we note $\vv_{||}=v_{||}\vB=v_{||}(a\er+b \ez+c \ephi)$ 
with $a=\frac{1}{R}\dz \psi$, $b=-\frac{1}{R}\dr \psi$ and $c=\frac{F_0}{R}$. Using these notations we obtain
\begin{align*}
\rho \vv_{||}\cdot\nabla \vv_{||} &=\rho v_{||}\left(a\dr (v_{||}a)+b\dz (v_{||}a)+\frac{c}{R}\dphi (v_{||}a)-v_{||}\frac{c^2}{R}\right)\er\\
& \quad+\rho v_{||}\left(a\dr (v_{||}b)+b\dz (v_{||}b)+\frac{c}{R}\dphi (v_{||}b)\right)\ez\\
&\quad +\rho v_{||}\left(a\dr (v_{||}c)+b\dz (v_{||}c)+\frac{c}{R}\dphi (v_{||}c)\right)\ephi+v_{||}a\frac{c}{R}\ephi.
\end{align*}
Now we rewrite the term as $\vB\cdot (\rho\vv_{||}\cdot\nabla \vv_{||})=W_1+W_2 +W_3$, where $W_1$ is given by
\begin{align*}
W_1&=\rho \frac{c}{R}v_{||}\left[a \dphi (v_{||}a)+b \dphi (v_{||}b)+c \dphi (v_{||}c)\right]+\\
&=\rho \frac{F_0}{2 R^2}\dphi\left(v_{||}^2 a^2+v_{||}^2 b^2+v_{||}^2 c^2\right )=\rho \frac{F_0}{R^2}\dphi\left(\frac{v_{||}^2|\vB|^2}{2}\right).
\end{align*}
The term $W_2$ is given by
\begin{align*}
W_2&=b \rho v_{||}\left[a\dz (v_{||} a)+b \dz (v_{||}b)+c\dz (v_{||}c)\right]\\
&=b\rho \dz\frac12\left(v_{||}^2 a^2+v_{||}^2 b^2+v_{||}^2 c^2\right )=-\frac{\rho}{R}\dr \psi \dz\left(\frac{v_{||}^2|\vB|^2}{2}\right).
\end{align*}
The term $W_3$ is given by
\begin{align*}
W_3 &=a \rho v_{||}\left[a \dr (v_{||}a)+b \dr (v_{||}b)+c \dr (v_{||}c)\right]\\
&=a \rho \dr\frac12\left(v_{||}^2 a^2+v_{||}^2 b^2+v_{||}^2 c^2\right )=\frac{\rho}{R}\dz \psi \dr\left(\frac{v_{||}^2|\vB|^2}{2}\right).
\end{align*}
At the end we obtain 
\begin{equation}\label{vp3}
\vB\cdot (\rho\vv_{||}\cdot\nabla \vv_{||})=-\frac{\rho}{R}[\psi,\frac{v_{||}^2|\vB|^2}{2}]+\rho\frac{F_0}{R^2}\dphi\left(\frac{v_{||}^2|\vB|^2}{2}\right).
\end{equation}
Now we propose to study the fourth term $\vB \cdot \left(\rho\mathbf{v}_{pol}\cdot\nabla \mathbf{v}_{pol}\right)$. To estimate this term we define $\alpha=-R \dz u $ and $\beta=R\dr u $. Using this notation we prove that
$$
\vv_{pol}= \alpha \er+\beta \ez\mbox{ and } \vv_{pol} \cdot\nabla \vv_{pol} =(\alpha \dr \alpha+\beta \dz \alpha)\er+(\alpha \dr \beta+\beta \dz \beta)\ez.
$$
Using the definitions of the coefficients we obtain
$$
\vB\cdot(\rho\mathbf{v}_{pol}\cdot\nabla \mathbf{v}_{pol})=\frac{\rho}{R}\rho\left(\dz(\psi)(\alpha \dr \alpha+\beta \dz \alpha)-\dr(\psi)(\alpha \dr \beta+\beta \dz \beta)\right)
$$
which is equal to
$$
\vB\cdot(\rho\mathbf{v}_{pol}\cdot\nabla \mathbf{v}_{pol})=\frac{1}{R}\left(\dz(\psi)(\dr(\frac12\alpha^2+\frac12\beta^2)+\beta(\dz \alpha-\dr \beta))-\dr(\psi)(\dz(\frac12\alpha^2+\frac12\beta^2)-\alpha(\dz \alpha-\dr \beta))\right)
$$
to obtain
$$
\vB\cdot(\rho\mathbf{v}_{pol}\cdot\nabla \mathbf{v}_{pol})=\frac{1}{2R}\rho[R^2|\nabla_{pol}u|^2,\psi]+\frac{\rho}{R}\left[\dz(\psi)\beta(\dz \alpha-\dr \beta)+\dr(\psi)\alpha(\dz \alpha-\dr \beta)\right].
$$
After straightforward computations we obtain
\begin{equation}\label{vp4}
\vB\cdot(\rho\mathbf{v}_{pol}\cdot\nabla \mathbf{v}_{pol})=\frac{1}{2R}\rho[R^2|\nabla_{pol}u|^2,\psi]+\rho R w [\psi,u].
\end{equation}
 Now we consider the term $\vB \cdot (\rho \vv_{pol}\cdot\nabla (v_{||}\vB))$. To estimate this term we define $\alpha=-R\dz u$, $\beta =R\dr u$, $a=\frac{1}{R}\dz \psi$, $b=-\frac{1}{R}\dr \psi$ and $c=\frac{F_0}{R}$. Consequently we obtain
$$
\vv_{pol}=\alpha \er +\beta \ez,\mbox{ and } v_{||}\vB=v_{||}(a \er + b\ez + c\ephi).
$$
Using  these notations we obtain
\begin{align*}
\rho \vv_{pol}\cdot \nabla (v_{||}\vB) =& +\rho[\alpha \dr (v_{||}a)+\beta \dz (v_{||}a)]\er\\
 & +\rho[\alpha \dr (v_{||}b)+\beta \dz (v_{||}b)]\ez+\rho[\alpha \dr (v_{||}c)+\beta \dz (v_{||}c)]\ephi,
\end{align*}
consequently 
\begin{align*}
\vB\cdot (\rho \vv_{pol}\cdot \nabla (v_{||}\vB)) = & +\rho a[\alpha \dr (v_{||}a)+\beta \dz (v_{||}a)]\\
&+\rho b[\alpha \dr (v_{||}b)+\beta \dz (v_{||}b)]+\rho c[\alpha \dr (v_{||}c)+\beta \dz (v_{||}c)].
\end{align*}
Rearranging terms we obtain
\begin{align*}
\vB\cdot (\rho \vv_{pol}\cdot \nabla (v_{||}\vB)) = &+\rho(a^2+b^2+c^2)\alpha\dr(v_{||})+\rho(a^2+b^2+c^2)\beta\dz(v_{||})\\
& +\frac12 \rho \alpha v_{||}\dr (a^2+b^2+c^2)+\frac12 \rho \beta v_{||}\dz (a^2+b^2+c^2).
\end{align*}
Using that $(a^2+b^2+c^2)=|\vB|^2$ we obtain that
\begin{equation}\label{vp5}
\vB\cdot (\rho \vv_{pol}\cdot \nabla (v_{||}\vB))=R\rho|\vB|^2[u,v_{||}]+R\rho v_{||}[u,\frac{|\vB|^2}{2}].
\end{equation}
To finish we consider the term $\vB \cdot (\rho v_{||}\vB\cdot\nabla (\vv_{pol}))=\rho v_{||}(\vB\cdot(\vB\cdot \nabla \vv_{pol}))$.
We define  $\vv_{pol}=\alpha \er+\beta \ez$ and $\vB=a \er + b\ez +c\ephi$. Using these definitions we obtain
\begin{align*}
\vB\cdot(\vB\cdot \nabla \vv_{pol})&=a\left[a\dr \alpha +b\dz \alpha +\frac{c}{R}\dphi \alpha\right]+b\left[a\dr \beta +b\dz \beta +\frac{c}{R}\dphi \beta\right]+\frac{c^2}{R}\alpha,\\
&=a\left[a\dr \alpha +b\dz \alpha\right]+b\left[a\dr \beta +b\dz \beta\right] +\frac{ac}{R}\dphi \alpha+\frac{bc}{R}\dphi \beta+\frac{c^2}{R}\alpha.
\end{align*}
Now we consider the first term $A=a(a\dr \alpha +b\dz \alpha)+b(a\dr \beta +b\dz \beta)$. For this we rewrite the term in the following form
$$
A=a\dr (a\alpha+b\beta)+b\dz (a\alpha+b\beta)-\alpha\dr(\frac{a^2}{2})-\beta\dz(\frac{b^2}{2})-\beta a\dr b-\alpha b \dz a.
$$
We define $C=a\dr (a\alpha+b\beta)+b\dz (a\alpha+b\beta)$ and $D=-\alpha\dr(\frac{a^2}{2})-\beta\dz(\frac{b^2}{2})-\beta a\dr b-\alpha b \dz a$.
We can rewrite the term $D$ in the following form 
$$
D=-(\alpha\dr(\frac{a^2}{2}+\frac{b^2}{2})+\beta\dz(\frac{a^2}{2}+\frac{b^2}{2})+\alpha b(\dz a-\dr b)-\beta a (\dz a-\dr b)).
$$
We obtain
$$
D=-R[u,\frac{|\vB_{pol}|^2}{2}]+\frac{j}{R}[u,\psi]
$$
Straightforward computations show that $C=\frac{1}{R}[\psi,(\gradp \psi\cdot \gradp u)]$. The term $A$ is given by $A=C+D$. Now we consider the term $B=\frac{ac}{R}\dphi \alpha+\frac{bc}{R}\dphi \beta+\frac{c^2}{R}\alpha$ and it is easy to prove that \\$B=-\frac{F_0}{R^2}(\gradp \psi\cdot \gradp (\dphi u))-\frac{F_0^2}{R^2}\dz u$.
At the end we obtain
\begin{align*}
\vB \cdot (\rho v_{||}\vB\cdot\nabla (\vv_{pol})) & =-R\rho v_{||}[u,\frac{|\vB_{pol}|^2}{2}]+\rho v_{||}\frac{j}{R}[u,\psi]+\frac{\rho v_{||}}{R}[\psi,(\gradp \psi\cdot \gradp u)]\\
& \quad -\frac{\rho v_{||}F_0}{R^2}(\gradp \psi\cdot \gradp (\dphi u))-\rho v_{||}\frac{F_0^2}{R^2}\dz u.\numberthis \label{vp6}
\end{align*}

\subsection{Final Model \label{sec:final_model}} 
We define the magnetic and velocity fields by $\vB=\frac{F_0}{R}\ephi+\frac{1}{R}\nabla \psi \times \ephi$ and $\vv=-R\nabla u \times \ephi+v_{||}\vB$.
Using all the equations (\ref{psifinal}), (\ref{u2},\ref{u3},\ref{u5},\ref{u6},\ref{u8},\ref{u12},\ref{u13}), (\ref{rho}), (\ref{T}), (\ref{dtvpar1},\ref{dtvpar2},\ref{vp1},\ref{vp3},\ref{vp4},\ref{vp5},\ref{vp6}) based on these definitions of the fields and the definition of the toroidal current and poloidal vorticity, we obtain the final reduced MHD model with parallel velocity. 
$$
\left\{\begin{array}{l}
\ds\dt\frac{1}{R^2}\psi=\frac{1}{R}[\psi,u]+\frac{\eta}{R^2}\left(j +\frac{\partial_{\phi\phi} \psi}{R^2}\right)-\frac{1}{R^2}F_0\dphi u,\\
\\
\\
\ds\nabla \cdot(\widehat{\rho}\nabla_{pol} \dt u) -\nabla \cdot (\rho \dt (v_{||}\gradp \psi))=\frac{1}{2R}[R^2|\nabla_{pol}u|^2,\widehat{\rho}]+\frac{1}{R}[\widehat{\rho} R^2w,u]-\frac{1}{R}[R^2,p]+\frac{1}{R}[\psi,j]-\frac{F_0}{R^2}\dphi j\\
\\
\ds+\nabla\cdot (\nu_{perp}\nabla w)+\frac{1}{R}[\rho v_{||}^2 j,\psi]+\frac{1}{R}[\rho v_{||}(\gradp v_{||}\cdot\gradp \psi),\psi]- \nabla \cdot \left(\widehat{\rho}\frac{F_0}{R^2}v_{||}\nabla_{pol} (\dphi u)\right)\\
\\
\ds+\frac{1}{R}[\widehat{\rho},v_{||}(\gradp \psi\cdot\gradp u)]-\frac{1}{2R}[\widehat{\rho},v_{||}^2|\vB_{pol}|^2]-\frac{1}{R}[\widehat{\rho}v_{||} w,\psi]+\frac{1}{R}[u,\widehat{\rho}v_{||} j ]+\frac{1}{R}[u,\widehat{\rho}(\gradp \psi\cdot \gradp v_{||})]\\
\\
\ds+\frac{1}{R}\dz\left(\widehat{\rho}\frac{F_0}{R^3}v_{||}\dphi(v_{||}\dz \psi)\right)+\frac{1}{R}\dr\left(\widehat{\rho}\frac{F_0}{R^3}v_{||}\dphi(v_{||}\dr \psi)\right)-\frac{1}{R}\dz\left(\widehat{\rho}\frac{F_0^2}{R^3} v_{||}^2\right),
\\
\\
\ds w=\lp u,\\
\\
\\
\ds j=\gs \psi,\\
\\
\\
\ds\dt\rho=R[\rho,u]+2\rho \dz u-\frac{v_{||} F_0}{R^2}\dphi \rho-\frac{v_{||}}{R}[\rho,\psi]-\frac{ \rho}{R}[v_{||},\psi]-\frac{\rho F_0}{R^2}\dphi v_{||},\\
\\
\\
\ds\dt p=R[p,u]+2\gamma p \dz u-\frac{v_{||} F_0}{R^2}\dphi p-\frac{v_{||}}{R}[p,\psi]-\frac{\gamma p}{R}[v_{||},\psi]-\frac{\gamma p F_0}{R^2}\dphi v_{||},\\
\\
\\
\rho|\vB^2|\dt v_{||}+\rho v_{||}\frac{1}{R^2}\gradp \psi \cdot\gradp (\dt \psi)-\rho\gradp \psi\cdot \gradp (\dt u)=-\frac{1}{R}[p,\psi]-\frac{F_0}{R^2}\dphi p+\frac{\rho }{R}[\psi,\frac{v_{||}^2|\vB|^2}{2}]\\
\\
\ds-\frac{F_0}{R}\rho \dphi\left(\frac{v_{||}^2|\vB|^2}{2R}\right)-\frac{1}{2R}\rho[R^2|\nabla_{pol}u|^2,\psi]-\frac{\widehat{\rho}}{R} w [\psi,u]-R\rho|\vB|^2[u,v_{||}]-R\rho v_{||}[u,\frac{|\vB|^2}{2}]+R\rho v_{||}[u,\frac{|\vB_{pol}|^2}{2}]\\
\\
\ds-\rho v_{||}\frac{\gs \psi}{R}[u,\psi]-\frac{\rho v_{||}}{R}[\psi,(\gradp \psi\cdot \gradp u)]+\frac{\rho v_{||}F_0}{R^2}(\gradp \psi\cdot \gradp (\dphi u))+\rho v_{||}\frac{F_0^2}{R^2}\dz u.\end{array}\right.
$$
In our derivation we have not treated the viscosity term $\triangle \vv$. This term in the resistive MHD is not really physical. This a very simple approximation of the stress tensor in the fluid model, which is physically justified for a gas but not for a magnetized plasma in a tokamak. It is used in JOREK  to model somewhat the effect of the stress tensor, dissipate the  energy and stabilize the system. For this reason, we propose to use a simple viscosity in the poloidal velocity equation given by $\nu \lp w=\nu \lp^2 u$  rather than compute the reduced viscosity associated to the viscosity $\triangle \vv$. We will discuss  the effects of this simplification on the total energy later.

\subsection{Energy estimate}
For the full MHD model the total energy is conserved in the ideal case and dissipated in the resistive case. To validate the derivation of the model, to validate the choice of the projection operators and to obtain the stability results, which are important for the numerical methods we prove that the reduced MHD model satisfies an energy balance equation compatible with the energy balance associated with full MHD model \cite{redMHD2}-\cite{kruger2}.
Before the energy estimate we introduce the natural Dirirchlet and Neumann boundary conditions \cite{redMHD} given by
\begin{equation}\label{bc}
\psi=u=T=\rho=0\mbox{ on }\partial \Omega\mbox{ and } \frac{\partial u}{\partial \mathbf{n}}=\frac{\partial \psi}{\partial \mathbf{n}}=0
\end{equation}
with $\mathbf{n}$ the outgoing normal to the domain. We can also use the boundary conditions 
\begin{equation}\label{bc2}
\psi=u=T=\rho=0\mbox{ on }\partial \Omega\mbox{ and } w=j=0
\end{equation}
These two boundary conditions are relatively close (see \cite{redMHD}).
\begin{lemma}
We define the energy $E=\frac{|\vB|^2}{2}+\rho\frac{|\vv|^2}{2}+\frac{1}{\gamma-1}p$.We assume that the boundary conditions are given by (\ref{bc}). If $\eta=\nu=0$ the total energy satisfies
$$
\frac{d}{dt}\int_{\Omega} E dW=\frac{d}{dt}\int_{\Omega} R E dV=0
$$
and if $\eta\neq0$ and $\nu\neq0$
$$
\frac{d}{dt}\int_{\Omega} E dW=-\nu\int_{\omega}w^2 dW-\eta\int_{\omega} j^2 dW- \eta \int_{\Omega}|\nabla_{pol}\left(\frac{\partial_{\phi} \psi}{R^2}\right)|^2dW.
$$
\end{lemma}
\begin{proof}
To begin we compute $\frac{dE}{dt}=\frac{d}{dt}\int_{\Omega} \frac{|\vB|^2}{2}+\rho\frac{|\vv|^2}{2}+\frac{1}{\gamma-1}p dW$. We obtain
\begin{align*}
\frac{dE}{dt} & =\int_{\Omega}\dt\frac{|\nabla_{pol} \psi|^2}{2R^2}dW+\int_{\Omega}\rho\dt \frac{|\vv_{pol}|^2}{2}dW+\int_{\Omega}\rho\dt \frac{|\vv_{||}|^2}{2}dW+\int_{\Omega}\rho (\vv_{||}\cdot\dt \vv_{pol}+\vv_{pol}\cdot\dt \vv_{||})dW\\
& \quad +\int_{\Omega}\frac{|\vv_{pol}|^2}{2}\dt \rho dW+\int_{\Omega}\frac{|\vv_{||}|^2}{2}\dt \rho dW+\int_{\Omega}(\vv_{||}\cdot \vv_{pol})\dt \rho dW+\int_{\Omega}\frac{\dt p}{\gamma -1}dW.
\end{align*}
After straightforward computations we shows that the derivative of the energy is given by
\begin{align*}
\frac{dE}{dt} & =\int_{\Omega}\dt\left(\frac{|\gradp \psi|^2}{2R^2}\right)dW+\int_{\Omega}\widehat{\rho}\dt \left(\frac{|\gradp u|^2}{2}\right)dW+\int_{\Omega}\rho |\vB|^2\dt \left(\frac{v_{||}^2}{2}\right)dW\\
& \quad+\int_{\Omega}\frac{\rho v_{||}^2}{R^2}(\gradp \psi\cdot\gradp (\dt \psi))dW\quad-\int_{\Omega}\rho(\gradp u\cdot\dt(v_{||}\gradp \psi))dW-\int_{\Omega}\rho v_{||}(\gradp \psi\cdot\dt(\gradp u))dW\\
& \quad + \int_{\Omega}\frac{|\nabla_{pol} u|^2}{2}\dt \widehat{\rho}dW+ \int_{\Omega}\frac{v_{||}^2|\vB|^2}{2}\dt \rho dW -\int_{\Omega}v_{||}(\gradp u\cdot\gradp \psi)\dt \rho dW  +\int_{\Omega}\frac{\dt p}{\gamma -1} dW.
\end{align*}
The term $\int_{\Omega}\dt\left(\frac{|\gradp \psi |^2}{2R^2}\right)dW$ is equal to $\int_{\Omega}(\frac{\gradp \psi}{R}\cdot\gradp (\dt \psi))dW $.
Integrating by parts we obtain 
$$
\int_{\Omega}\dt\left(\frac{|\gradp \psi|^2}{2R^2}\right)dW=\int_{\Omega}\dt\left(\frac{|\gradp \psi|^2}{2R}\right)dV=-\int_{\Omega}\frac{\gs \psi}{R}\dt \psi dV=-\int_{\Omega}\frac{j}{R^2}\dt \psi dW.
$$
Using an integration by parts we also obtain 
$$
\int_{\Omega}\widehat{\rho}\dt \left(\frac{|\gradp u|^2}{2}\right)=-\int_{\Omega} \nabla \cdot \left(\widehat{\rho}\gradp (\dt u)\right) u dW.
$$
Consequently 
\begin{align*}
\frac{dE}{dt} &=-\int_{\Omega} (\dt \psi)\frac{j}{R^2}dW-\int_{\Omega}\left(\nabla \cdot(\widehat{\rho}\gradp (\dt u))-\nabla \cdot (\rho \dt (v_{||}\gradp \psi))\right)u dW\\
&\quad+\int_{\Omega}\rho\left(|\vB|^2\dt v_{||}+\frac{v_{||}}{R^2}(\gradp \psi\cdot\gradp (\dt \psi))-(\gradp \psi \cdot \dt (\gradp u))\right)v_{||} dW\\
&\quad + \int_{\Omega}\frac{|\gradp u|^2}{2}\dt \widehat{\rho}dW+ \int_{\Omega}\frac{v_{||}^2|\vB|^2}{2}\dt \rho dW-\int_{\Omega}v_{||}(\gradp u\cdot\gradp \psi)\dt \rho dW +\int_{\Omega}\frac{\dt p}{\gamma -1}dW.
\end{align*}
Using $2\dz u=\frac{1}{R}[R^2,u]$, we obtain
$$
\dt\rho=R[\rho,u]+\frac{\rho}{R}[R^2,u]-\frac{v_{||} F_0}{R^2}\dphi \rho-\frac{v_{||}}{R}[\rho,\psi]-\frac{ \rho}{R}[v_{||},\psi]-\frac{\rho F_0}{R^2}\dphi v_{||}.
$$
Before computing the energy estimate we give an equation on $\widehat{\rho}$.
For $\widehat{\rho}$ we multiply by $R^2$ the equation on $\rho$. We obtain
\begin{align*}
\dt \widehat{\rho}= & R^3[\rho,u]+2\widehat{\rho}\dz u+R[\psi,\rho v_{||}]-F_0\dphi(\rho v_{||})\\
= &R[\widehat{\rho},u]-\frac{\widehat{\rho}}{R}[R^2,u]+2\widehat{\rho}\dz u+R[\psi,\rho v_{||}]-F_0\dphi(\rho v_{||}).
\end{align*}
Using that $2\widehat{\rho}\dz u=\frac{\widehat{\rho}}{R}[R^2,u]$, we obtain
$$
\dt \widehat{\rho}=R[\widehat{\rho},u]+R[\psi,\rho v_{||}]-F_0\dphi(\rho v_{||}),
$$
which is equal to
$$
\dt \widehat{\rho}=R[\widehat{\rho},u]+R\rho[\psi, v_{||}]+R v_{||}[\psi,\rho ]-F_0\rho\dphi v_{||}-F_0v_{||}\dphi\rho.
$$
To compute $\frac{dE}{dt}$ we add to the final model three equations on the density:
$$
\left\{\begin{array}{l}
\ds\frac{v_{||}^2|\vB|^2}{2}\dt\rho=\frac{v_{||}^2|\vB|^2}{2R}[\widehat{\rho},u]-\frac{v_{||}^2|\vB|^2}{2}\frac{ F_0}{R^2}\dphi (\rho v_{||})-\frac{v_{||}^2|\vB|^2}{2}\frac{1}{R}[\rho v_{||},\psi],\\
\\
\ds v_{||}(\gradp u\cdot\gradp \psi)\dt\rho=\frac{v_{||}}{R}(\gradp u\cdot\gradp \psi)[\widehat{\rho},u]-v_{||}(\gradp u\cdot\gradp \psi)\frac{ F_0}{R^2}\dphi (\rho v_{||})-v_{||}(\gradp u\cdot\gradp \psi)\frac{1}{R}[\rho v_{||},\psi],\\
\\
\ds \frac{|\nabla_{pol} u|^2}{2}\dt\widehat{\rho}=\frac{|\nabla_{pol} u|^2}{2}R[\widehat{\rho},u]-\frac{|\nabla_{pol} u|^2}{2}F_0\dphi (\rho v_{||})-\frac{|\nabla_{pol} u|^2}{2}R[\rho v_{||},\psi].\end{array}\right.
$$
Now we compute
\begin{align*}
\frac{dE}{dt} & =-\int_{\Omega} (\dt \psi)\frac{j}{R^2}dW-\int_{\Omega}\left(\nabla \cdot(\widehat{\rho}\gradp (\dt u))-\nabla \cdot(\rho \dt (v_{||}\gradp \psi))\right)u dW\\
&\quad +\int_{\Omega}\rho\left(|\vB|^2\dt v_{||}+\frac{v_{||}}{R^2}(\gradp \psi\cdot\gradp (\dt \psi))-(\gradp \psi\cdot\dt (\gradp u))\right)v_{||} dW\\
&\quad+ \int_{\Omega}\frac{|\nabla_{pol} u|^2}{2}\dt \widehat{\rho}dW+ \int_{\Omega}\frac{v_{||}^2|\vB|^2}{2}\dt \rho dW-\int_{\Omega}v_{||}(\gradp u\cdot\gradp \psi)\dt \rho dW +\int_{\Omega}\frac{\dt p}{\gamma -1}dW.
\end{align*}
The derivative in time $\frac{dE}{dt}$ can by writing as the sum of 18 groups of terms: $\frac{dE}{dt}=(E1)+...+(E18)$. Now we propose to prove that each group of terms is equal to zero or negative:
\begin{align*}
(E1) &=-\int_{\Omega}\frac{1}{R}[\psi,u] j dW-\int_{\Omega}\frac{1}{R}[\psi,j]u dW=-\int_{\Omega}([\psi,u] j+[\psi,j]u) dV=0,\\
(E2) &=-\int_{\Omega}\frac{F_0}{R^2}\dphi (u) j dW+\int_{\Omega}\frac{F_0}{R^2}\dphi(j) u dW=\int_{\Omega}\frac{F_0}{R}(\dphi (u) j+ \dphi(j)u) dV=0.
\end{align*}
These results are obtained by integration by parts and using the assumptions on the boundary conditions (\ref{bc}). Now we study the term $(E3)$:
$$
(E3)=-\int_{\Omega}\frac{1}{R}[\widehat{\rho}R^2 w,u]u dW=-\int_{\Omega}[\widehat{\rho}R^2 w,u]u dV=\int_{\Omega}[u,u]\widehat{\rho}R^2 w dV=0.
$$
The term $(E4)$ corresponds to the viscosity and resistivity terms:
\begin{align*}
(E4) & =-\nu\int_{\Omega}\lp w u dW-\int_{\Omega} \eta \frac{j^2}{R^2} dW-\eta \int_{\Omega}\frac{\partial_{\phi\phi} \psi}{R^4} j dW\\
& =-\nu\int_{\Omega} w^2 dW-\int_{\Omega} \eta \frac{j^2}{R^2}dW-\eta \int_{\Omega}|\nabla_{pol}\left(\frac{\partial_{\phi} \psi}{R^2}\right)|^2dW.
 \end{align*}
To obtain this result we use $w=\lp u$ and a double integrating by parts. Now we define the term $(E5)$ which depends on the pressure.
$$
(E5)=\int_{\Omega}\frac{1}{R}[R^2,p]u dW+\frac{1}{\gamma-1}\int_{\Omega}R[p,u]dW+\frac{2\gamma}{\gamma-1}\int_{\Omega}p \dz u dW.
$$
Using $2p\dz u=\frac{p}{R}[R^2,u]$ and integrating by parts we obtain
\begin{align*}
(E5) &=\int_{\Omega}[R^2,p]u dV+\frac{1}{\gamma-1}\int_{\Omega}R^2[p,u]dV+\frac{\gamma}{\gamma-1}\int_{\Omega}p [R^2,u] dV.\\
(E5) &=-\int_{\Omega}[R^2,u]p dV-\frac{1}{\gamma-1}\int_{\Omega}p[R^2,u]dV+\frac{\gamma}{\gamma-1}\int_{\Omega}p [R^2,u] dV=0.
\end{align*}
Now we study the terms $(E6)$ and $(E7)$. In these two cases using integration by parts and the anti-symmetric properties of bracket operator we conclude.
\begin{align*}
(E6) &=-\int_{\Omega}\frac{1}{2R}[R^2|\nabla_{pol} u|^2,\widehat{\rho}]u dW+\int_{\Omega}\frac{|\nabla_{pol} u|^2}{2}R[\widehat{\rho},u]dW,\\
 & =-\int_{\Omega}\frac{1}{2}[R^2|\nabla_{pol} u|^2,\widehat{\rho}]u dV-\int_{\Omega}u[\widehat{\rho},\frac{|\nabla_{pol} u|^2}{2}R^2]dV=0,
\end{align*}
\begin{align*}
(E7) =& +\int_{\Omega}\frac{1}{R}[\widehat{\rho}v_{||}w,\psi]u dV-\int_{\Omega}\frac{\widehat{\rho}}{R}w [\psi,u]v_{||}dW.\\
= &-\int_{\Omega}\widehat{\rho}v_{||}w [u,\psi] dV-\int_{\Omega}\widehat{\rho} w [\psi,u]v_{||}dV=0.
\end{align*}
 The term $(E8)$ correspond to the coupling between the pressure and the parallel velocity $v_{||}$. We obtain
\begin{align*}
(E8) & =-\frac{1}{\gamma-1}\int_{\Omega}\frac{F_0v_{||}}{R^2}(\dphi p)dW-\frac{1}{\gamma-1}\int_{\Omega}\frac{v_{||}}{R}[p,\psi]dW-\frac{\gamma}{\gamma-1}\int_{\Omega}\gamma\frac{p}{R}[v_{||},R]dW\\
& \quad -\frac{\gamma}{\gamma-1}\int_{\Omega}\frac{F_0}{R^2}p(\dphi v_{||})dW-\int_{\Omega}\frac{F_0v_{||}}{R^2}(\dphi p)dW-\int_{\Omega}\frac{v_{||}}{R}[p,\psi]dW.
\end{align*}
Integrating by parts the terms which depend on $\frac{\gamma}{\gamma-1}$ and factorizing we obtain that $(E8)=0$. The term $(E9)$ is defined by 
\begin{align*}
(E9)& = + \int_{\Omega}\frac{v_{||}^2|\vB_{pol}|^2}{2}\frac{1}{R}[\rho v_{||},\psi]dW+\int_{\Omega}\frac{\rho v_{||}}{R}[\psi,\frac{v_{||}^2|\vB_{pol}|^2}{2}]dW\\
&\quad -\int_{\Omega}\frac{v_{||}^2|\vB_{pol}|^2}{2}\frac{F_0}{R^2}\dphi (\rho v_{||})dW-\int_{\Omega}\frac{F_0}{R^2}\rho v_{||}\dphi\left(\frac{v_{||}^2|\vB_{pol}|^2}{2}\right)dW.
\end{align*}
Integrating by parts we obtain
\begin{align*}
(E9) & = +\int_{\Omega}\rho v_{||}\left([\psi,\frac{v_{||}^2|\vB_{pol}|^2}{2}]+[\frac{v_{||}^2|\vB_{pol}|^2}{2},\psi]\right)dV\\
&\quad +\int_{\Omega}\frac{F_0}{R}\rho v_{||}\dphi\left(\frac{v_{||}^2|\vB_{pol}|^2}{2}\right)dV-\int_{\Omega}\frac{F_0}{R}\rho v_{||}\dphi\left(\frac{v_{||}^2|\vB_{pol}|^2}{2}\right)dV=0.
\end{align*}
The term $(E10)$ is defined  by
$$
(E10)=-\int_{\Omega}\frac{|\nabla_{pol} u|^2}{2}R[\rho v_{||},\psi]dW-\int_{\Omega}\frac{1}{2R}\rho v_{||}[R^2|\nabla_{pol} u|^2,\psi]dW.
$$
We apply the classical integration to conclude (E10)=0. Now we study the term $(E11)$
$$
(E11)=-\int_{\Omega}\frac{1}{R}[u,\widehat{\rho}v_{||} j ]u dW-\int_{\Omega}\frac{1}{R}[u,\widehat{\rho}(\gradp \psi\cdot\gradp v_{||})]u dW.
$$
To conclude we use the integration by parts and the fact that $[u,u]=0$. The term $(E12)$ depends of the toroidal direction
\begin{align*}
(E12) =& +\int_{\Omega}\nabla \cdot\left(\widehat{\rho}\frac{F_0}{R^2}v_{||}\gradp (\dphi u)\right)dW-\int_{\Omega}\frac{|\nabla_{pol} u|^2}{2}F_0\dphi (\rho v_{||})\\
 = & -\int_{\Omega}\widehat{\rho}\frac{F_0}{R}v_{||}\left(\gradp u \cdot\gradp (\dphi u)\right)dV-\int_{\Omega}\frac{|\nabla_{pol} u|^2}{2}F_0\dphi \left(\frac{\widehat{\rho}}{R} v_{||}\right).
\end{align*}
Using that $(\gradp u\cdot\gradp(\dphi u))=\dphi(\frac{|\nabla_{pol} u|^2}{2})$ and integrating by parts the second term we conclude.
The term $(E13)$ also depends on the toroidal derivative. It is defined by
\begin{align*}
(E13) & =-\int_{\Omega}\frac{1}{R}\dz\left(\widehat{\rho}\frac{F_0}{R^3}v_{||}\dphi(v_{||}\dz \psi)\right)dW-\int_{\Omega}\frac{1}{R}\dr\left(\widehat{\rho}\frac{F_0}{R^3}v_{||}\dphi(v_{||}\dr \psi)\right)dW\\
& \quad+\int_{\Omega}v_{||}(\gradp u\cdot\gradp \psi)\frac{F_0}{R^2}\dphi(\rho v_{||})dW+\int_{\Omega}\frac{\rho v_{||}^2}{R^2}F_0 \left(\gradp \psi\cdot\gradp (\dphi u)\right)dW.
\end{align*}
We integrate by parts the first term and expand this term, we integrate by parts the second one to obtain
\begin{align*}
(E13) &=\int_{\Omega}\rho  \frac{F_0}{R}v_{||}(\dphi(\gradp \psi)\cdot\gradp u) dV+\int_{\Omega}\rho v_{||}\frac{F_0}{R}(\gradp u\cdot\gradp \psi)\dphi(v_{||})dV\\
& \quad -\int_{\Omega}\frac{F_0\rho v_{||}}{R} \dphi (v_{||}(\gradp u\cdot\gradp \psi))dV+\int_{\Omega}\frac{\rho F_0}{R}v_{||}^2 (\gradp \psi\cdot\gradp (\dphi u))dV.
\end{align*}
To conclude we expand the third term $\frac{F_0\rho v_{||}}{R} \dphi (v_{||}(\gradp u\cdot\gradp \psi))$ in two terms $\frac{F_0\rho v_{||}^2}{R} \dphi (\gradp u\cdot\gradp \psi)$ and $\frac{F_0\rho v_{||}}{R}(\gradp u\cdot\gradp \psi) \dphi (v_{||})$. The sum of the five terms obtained is equal to zero. Now we introduce the terms $(E14)$ and $(E15)$.
\begin{align*}
(E14) &=-\int_{\Omega}\frac{1}{R}[\widehat{\rho},v_{||}(\gradp \psi\cdot\gradp u)]u dW-\int_{\Omega}\frac{1}{R}v_{||}(\gradp u\cdot\gradp \psi)[\widehat{\rho},u]dW,\\
(E15) &=-\int_{\Omega}\frac{1}{R}[\rho v_{||}^2 \gs \psi,\psi]u dW-\int_{\Omega}\rho v_{||}\frac{\gs \psi}{R}[u,\psi]dW.
\end{align*}
A integration by parts of the first term is sufficient to prove that $(E14)$ and $(E15)$ are equal to zero. The term (E16) is defined by 
\begin{align*}
(E16) = & +\int_{\Omega}\frac{1}{2R}[\widehat{\rho},v_{||}^2|\vB_{pol}|^2]u dW+\int_{\Omega}\frac{v_{||}^2|\vB|^2}{2R}[\widehat{\rho},u] dW\\
& -\int_{\Omega}R\rho v_{||}^2[u,\frac{|\vB|^2}{2}]dW-\int_{\Omega}R\rho v_{||}\frac{F_0}{R^2}[u,v_{||}]dW+\int_{\Omega}R\rho v_{||}^2[u,\frac{|\vB_{pol}|^2}{2}]dW.
\end{align*}
The fourth term of $(E16)$ is the toroidal part of the term $R\rho |\vB|^2[u,v_{||}]$ in the parallel velocity equation. Now we split $(E16)$ between two terms $(E16a)$ and $(E16b)$. $(E16a)$ is defined by
\begin{align*}
(E16a) =& +\int_{\Omega}\frac{1}{2R}[\widehat{\rho},v_{||}^2|\vB_{pol}|^2]u dW+\int_{\Omega}v_{||}^2\frac{|\vB_{pol}|^2}{2R}[\widehat{\rho},u]dW\\
& -\int_{\Omega}R\rho v_{||}^2 [u,\frac{|\vB_{pol}|^2}{2}]dW+\int_{\Omega}R\rho v_{||}^2 [u,\frac{|\vB_{pol}|^2}{2}]dW.
\end{align*}
This term is equal to zero (integrating by parts the first term is sufficient to prove this). The $(E16b)$ is defined by
$$
(E16b)=\int_{\Omega}\frac{v_{||}^2}{R}\frac{F_0^2}{2R^2}[\widehat{\rho},u]dW-\int_{\Omega}R\rho v_{||}^2[u,\frac{F_0^2}{2R^2}]dW
-\int_{\Omega}R\rho v_{||}\frac{F_0^2}{R^2}[u,v_{||}]dW
$$
We rewrite the term $(E16b)$ to obtain
\begin{align*}
(E16b) = & +\int_{\Omega}v_{||}^2\frac{F_0^2}{2R^2}[\widehat{\rho},u]dV-\int_{\Omega}\widehat{\rho} v_{||}^2[u,\frac{F_0^2}{2R^2}]dV\\
& -\int_{\Omega}\widehat{\rho} v_{||}\frac{F_0^2}{2R^2}[u,v_{||}]dV-\int_{\Omega}\widehat{\rho} v_{||}\frac{F_0^2}{2R^2}[u,v_{||}]dV
\end{align*}
We combine the second and third terms and use the anti-symmetry property of the bracket for the fourth term. We obtain
$$
(E16b)=-\int_{\Omega}\widehat{\rho} v_{||}[u,\frac{F_0^2}{2R^2}v_{||}]dV+\int_{\Omega}v_{||}^2\frac{F_0^2}{2R^2}[\widehat{\rho},u]dV+\int_{\Omega}\widehat{\rho} v_{||}\frac{F_0^2}{2R^2}[v_{||},u]dV.
$$
Now we combine the two last terms and we use anti-symmetric property of the bracket in the first to obtain
$$
(E16b)=\int_{\Omega}\widehat{\rho} v_{||}[\frac{F_0^2}{2R^2}v_{||},u]dV+\int_{\Omega}v_{||}\frac{F_0^2}{2R^2}[\widehat{\rho}v_{||},u]dV=0.
$$
The result is obtained using an integration by parts. The last $(E17)$ is given by
\begin{align*}
(E17)= &-\int_{\Omega}\frac{1}{R}[\rho v_{||}(\gradp v_{||}\cdot\gradp \psi),\psi] u dW+\int_{\Omega}\frac{1}{R}v_{||}(\gradp u\cdot\gradp v_{||})[\rho v_{||},\psi]dW,\\
&-\int_{\Omega}\frac{1}{R}\rho v_{||}^2[\psi,(\gradp \psi\cdot\gradp u)]dW-\int_{\Omega}R\rho|\vB_{pol}|^2[u,v_{||}]v_{||}dW.
\end{align*}
Firstly $R\rho|\vB_{pol}|^2[u,v_{||}]v_{||}=\frac{\rho}{R}(\gradp \psi\cdot\gradp \psi)v_{||}[u,v_{||}]$, secondly we have the identity 
$$
(\gradp v_{||}\cdot\gradp \psi)[u,\psi]=(\gradp \psi\cdot\gradp \psi)[u,v_{||}]+(\gradp u\cdot\gradp \psi)[v_{||},\psi].
$$
 Using these two identities we obtain
\begin{align*}
(E17) = & +\int_{\Omega}\rho v_{||}(\gradp v_{||}\cdot\gradp \psi)[u,\psi]dV-\int_{\Omega}\rho v_{||}^2[(\gradp u\cdot\gradp \psi),\psi]dV\\
& -\int_{\Omega}\rho v_{||}(\gradp u\cdot\gradp \psi)[v_{||},\psi]dV- \int_{\Omega}\rho v_{||}^2[\psi,(\gradp u\cdot\gradp \psi)]dV\\
& -\int_{\Omega}\rho v_{||}(\gradp \psi\cdot\gradp \psi)[u,v_{||}]dV.
\end{align*}
The sum of second and fourth terms is equal to zero (anti-symmetry property of the bracket). The sum of the other terms is equal to zero (second identity). 
To finish the proof we compute $(E18)$ defined by
$$
(E18)=\int_{\Omega} \frac{1}{R}\dz\left(\widehat{\rho}\frac{F_0^2}{R^3} v_{||}^2\right)u dW+\int_{\Omega}\rho v_{||}
^2\frac{F_0^2}{R^2}\dz u dW.
$$
This term is equal to zero because the sum of the two terms is also equal to zero (using a integration by parts).
This last result concludes the proof.
\end{proof}
This result proves that the physical energy associated with the reduced MHD system is conserved in the ideal case ($\nu=\eta=0$) as for the full MHD case and dissipated in the resistive case. As for the full MHD case the dissipation is linked to the vorticity and the current. However the dissipation terms are not exactly the same in the reduced and full MHD. In the part of the dissipation which depends on the resistive terms, we have the square of the current for the full MHD and the square of the toroidal current of the reduced MHD. Consequently during the reduction the poloidal current disappears (we can explain this by the choice of the projections during the reduction). The ordering proposed in the physics papers show that the poloidal current is smaller than the toroidal current, consequently it is logical that the reduction kills the effects of this part. In the part of the dissipation which depends on the viscous terms we observe that the part linked to compressibility (divergence of $\vv$) and the parallel vorticity disappears. At the end we conserve only the dissipation associated with the poloidal vorticity. Finally, in the ideal case the reduced model conserves the energy as for the full MHD problem and in the resistive and viscous cases the reduced model dissipates energy with  decay terms that are relatively close to the decay terms of the dissipation of the full MHD. First this result validates the reduced model since we obtain two consistent  energy balance estimates associated to the full and reduced MHD models. Secondly the dissipation result  is useful to verify at the mathematical level that the model is well-posed. For example in \cite{redMHD}-\cite{redMHD2}  the authors explain and detail the key role of the energy balance to prove the existence of weak solutions. Finally, this energy estimate is very important to ensure the numerical stability of the schemes. Indeed a way to ensure the stability is to design a numerical method which dissipates the energy at the discrete level and we cannot obtain this stability property a similar energy dissipation on the continuous model.

Let us make a first remark about the resistive term $ \frac{\partial_{\phi\phi} \psi}{R^4} $, which is the poloidal current neglected in the JOREK code. With or without this term we have a model which conserves energy in the ideal case and dissipate the energy in the resistive case. 
And a second remark about the other invariant of the MHD. The classical full MHD admits other quantities, which are conserved. The first invariant is the mass conservation. When we have written the equation on the density we have plugged our reduced velocity field and never used an approximation. Consequently we can write the density equation in a conservative form and obtain the mass conservation. The second invariant is the cross helicity which is conserved only in the incompressible case. In our case we assume that the flow is compressible consequently the cross helicity may not be conserved. After it is not clear that the balance law for the cross helicity is the same for the reduced and the full MHD. The last one is the magnetic helicity (conserved only when the resistivity is equal to zero) defined by $(\mathbf{A}\cdot \vB)$ with $\mathbf{A}$ the vector potential given by
$$
\mathbf{A}=\frac{1}{R}\psi\ephi.
$$
Consequently the equation on the magnetic helicity is the equation on $\psi$. Now using the same boundary conditions that before we obtain
\begin{align*}
(\mathbf{A}\cdot\vB) = & F_0\frac{d}{dt}\int_{\Omega}\frac{1}{R^2}\psi dW\\
=& F_0\int_{\Omega}[\psi,u]dV-F_0\int\frac{F_0}{R}\dphi u dV\\
=&F_0\int_{\Omega}(\dr(\psi\dz u)-\dz(\psi\dr u))dV-F_0\int\frac{F_0}{R}\dphi u dV\\
= & F_0\int_{\partial \Omega}\psi \nabla_{pol}u \cdot\mathbf{n} dV-F_0\int\frac{F_0}{R}\dphi u dV=0
\end{align*}
Consequently the magnetic helicity is conserved.
\section{Discretization of the model}
\subsection{Spatial discretization}
In the JOREK code, different discretization methods are applied for the toroidal direction and the poloidal plane. For the toroidal direction we use a classical Fourier expansion.  This discretization is easy to implement but generates a large matrix. Using a Fast Fourier transformation (FFT) we obtain a faster algorithm to construct the matrix and the right hand side than the classical loop used to assemble the matrix and the right hand side. For the Poloidal plane we use a classical finite element method with numerical viscosity to stabilize the method.  The elements chosen are Cubic Bezier elements which allow to guarantee $C^1$ continuity useful to discretize the fourth order operators and preserve the free divergence constraints. However this $C^1$ reconstruction is not guaranteed for the grid center and for the X-Point. Because of the higher continuity requirement, these elements need only 4 degrees of freedom per grid node compared to the Lagrangian $ \mathbb{Q}_3$ cubic element, which needs 9 degrees of freedom by grid node. Another advantage comes from the isoparametric formulation. Indeed we can discretize the geometrical quantities  like $R$ and $Z$ with Bezier Splines. This property allows to construct the  grid aligned with the magnetic surfaces easily. The details about the discretization using Bezier elements are given in \cite{huymans2}.

\subsection{Original time discretization and preconditioning}
In this section we explain the time discretization originally used in JOREK and the preconditioning used for the linear solver. The different models implemented in the JOREK code (with or without parallel velocity) can be written in the following form
$$
\dt A(\U)=B(\U)
$$
with $A$ and $B$ discrete nonlinear differential operators and $\U=(\psi,u,j,w,\rho,T,v_{||})$. For the time discretization we use the classical Crank Nicholson or a Gear second order scheme allowing to write the time scheme in the following form
$$
(1+\zeta)A(\U^{n+1})-\theta\Delta t B(\U^{n+1})=(1+2\zeta)A(\U^{n})-\zeta A(\U^{n-1})+(1-\theta)\Delta t B(\U^{n})
$$
with $\zeta$ and $\theta$ the parameters of the scheme. If $\theta=1$ and $\zeta=0$ we obtain the implicit Euler scheme, if $\zeta=0$ and $\theta=\frac12$ we obtain the Crank-Nicholson scheme and if $\theta = 1$ and $\zeta =\frac12$ we obtain the Gears scheme. These implicit schemes do not preserve the decay of the discrete time energy, because the system is too nonlinear. Finding an accurate time scheme with this property is an interesting problem for the future. 
Now we define two nonlinear vectors $G(\U)=(1+\zeta)A(\U)-\theta\Delta t B(\U)$ and $b(\U^{n},\U^{n-1})=(1+2\zeta)A(\U^{n})-\zeta A(\U^{n-1})+(1-\theta)\Delta t B(\U^{n})$. At the end we want to solve the following nonlinear system  
$$
G(\U^{n+1})=b(\U^n,\U^{n-1}).
$$
A first order linearization is applied in the original code to obtain the following linear system
$$
\left(\frac{\partial G(\U^{n})}{\partial\U}\right)\delta \U^{n+1}=-G(\U^n)+b(\U^n,\U^{n-1})=R(\U^n)
$$
with $\delta \U^{n+1}=\U^{n+1}-\U^{n}$ and the Jacobian $J_n=\frac{\partial G(\U^{n})}{\partial\U}$. To solve this system we use the classical GMRES method with left preconditioning \cite{precon1,precon2}.
The principle of the left preconditioning is to replace the solver $J_n\delta \U^{n+1}=R(\U^n)$ by $M_n^{-1}J_n\delta \U^{n+1}=M_n^{-1}R(\U^n)$. The last system can be split between two steps. First we solve exactly
$$
M_n \delta \mathbf{y}=R(\U^n)
$$
and then we solve with the GMRES method
$$
M_n^{-1}J_n\delta \U^{n+1}=\delta \mathbf{y}.
$$
It is necessary to obtain the final algorithm that the preconditioning matrix $M_n$ is invertible. The idea currently followed in the code is to write the Jacobian by block, each block corresponding to the coupling terms between two Fourier modes. Under the assumption of weak coupling it is possible to eliminate the non diagonal blocks. We obtain a diagonal block matrix where the blocks correspond to the equations for each Fourier mode. To compute the inverse we use a direct solver (LU method for example) to obtain the inverse of each block and consequently the inverse of $M_n$. To minimize the CPU cost we don't invert $M_n$ at each time step, but only when  the convergence for the previous linear step is too slow.

\subsection{Nonlinear time solvers}
The first order linearization previously used may not be the optimal choice to solve the problem in the nonlinear phase of the run. Consequently we propose to replace this linearization by a Newton procedure. Since we use an iterative solver to compute the solution of the linear system, it will be interesting to use an inexact Newton procedure \cite{IneNew1}-\cite{IneNew2}. This variation of the Newton method means that the convergence criterion of the GMRES method  is adapted using the nonlinear residual and the convergence of the Newton procedure. The aim is to use the nonlinear convergence to minimize the number of GMRES iterations. Indeed it is not necessary to solve with a high accuracy the linear system but just enough at each step to converge to the solution of the nonlinear system at the end. Let us now detail the Inexact Newton algorithm:
~\\

\textbf{Algorithm}
~\\
\begin{itemize}
\item At the time step $n$, we compute $b(\U^n,\U^{n-1})$, $G(\U^n)$.
\item We choose $\epsilon^{GMRES}$ and the initial guess $\delta \U_0$. 
\item At each iteration $k$ of the Newton method we have the solution $\U_k$.
\item We compute $G(\U_k)$ and the Jacobian $J_k$.
\item We solve the linear system with GMRES $J_k\delta \U_k=-G(\U_k)+b(\U^n,\U^{n-1})=R(\U_k,\U^n)$ and the following convergence criterion
$$
\frac{||J_k \delta \U_k - R(\U_k,\U^n)||}{||R(\U_k,\U^n) ||}\leq \epsilon_{GMRES}^k
$$
with
$$
\epsilon_{GMRES}^k=\gamma\left(\frac{||R(\U_k,\U^n)||}{|| R(\U_{k-1},\U^n)||}\right)^{\alpha}.
$$
\item We iterate with $\U_{k+1}=\U_{k}+\delta\U_k$. 
\item We apply a convergence test (for example $||R(\U_k,\U^n)||<\epsilon_{a}+\epsilon_{r}||R(\U^n)||$).
\item When the Newton method has converged we define $\U^{n+1}=\U_{k+1}$. 
\end{itemize}
Here $\epsilon_{a}$ and $\epsilon_{r}$ are the relative and absolute stopping criteria for the Newton procedure.
We couple this algorithm with an adaptive time stepping which allows to use large time steps in the linear phase and smaller time steps in the nonlinear phase. Actually the principle is simple: if the Newton process converges very quickly we increase the time step and if the convergence is slow we decrease the time step for the following iteration. If the Newton process does not converge or if $||R(\U_{k+1},\U^n)||>||R(\U_k,\U^n)||$ during two or three consecutive linear steps we decrease the time step and restart the Newton iterations. To have a smooth evolution of the time step it is necessary to avoid a large increasing or decreasing of the time step.

\section{Numerical results}
In general the different test cases used in this paper have the same structure. First we compute the equilibrium on the poloidal grid (Fig 1., left), compute the aligned grid (Fig 1., right) and begin the time loop. At the beginning of the time loop  peeling-ballooning modes \cite{snyder,holzl} set in which are responsible for the appearance of edge localized modes (ELMs). These linear instabilities are driven by large pressure gradients 
(steep pressure pedestal) and large current densities in the edge.
During these instabilities the energy associated with the non principal modes grows exponentially. The background profiles are modified.  When the energies associated with the non principal modes are sufficiently large, the pressure gradients get smaller which stabilizes the instability.  This is the nonlinear saturation phase. The implicit time methods are known to be stable without restriction on the time step, however this type of result is valid for stable physical dynamics and stable models. In our cases we have physical instabilities consequently the numerical stability is not ensured.
Typically we will show that if the numerical error (time error, linearization error)  becomes too large the numerical simulation does not capture correctly the beginning of the salutation phase and generates critical numerical instabilities.\\
~\\
In this section we present some numerical results for the models with and without parallel velocity. We add to the reduced MHD models, numerical diffusion operators for each equation and  two anisotropic diffusion operators on the density and the temperature (density and pressure equations). For example, for the pressure equation we add the following  diffusion operator
$$
\nabla \cdot (k_{||}\nabla_{||}T+k_{\perp}\nabla_{\perp}T)=\nabla \cdot ((k_{||}-k_{\perp})\nabla_{||}T+k_{\perp}\nabla T)
$$ 
with 
$\nabla_{||} T=\frac{\vB}{||\vB||}\cdot(\frac{\vB}{||\vB||}\cdot\nabla T)$ and $\nabla_{\perp}=\nabla - \nabla_{||} $.\\
We propose to compare the different methods (Exact and Inexact Newton methods and classical linearization)  mainly in the nonlinear phase. Indeed in the linear phase  the classical method is clearly more efficient. In this phase the preconditioning is very efficient and the GMRES solver converges quickly (between 1 and 5 iterations). The Newton procedure converges with 3 iterations in general. Consequently using the Newton method the cost is clearly higher for each time step in the linear phase. In the nonlinear phase the situation is more complicated. The nonlinear phase begins when the quantities associated with the non principal modes have the same order of magnitude as the quantities associated to the principal mode. To compare the numerical results, we define the beginning of the nonlinear phase as the time where the kinetic and magnetic energies for $n \neq 0$ are at the level of the energies associated to $n = 0$.
To compare the classical linearization and the Newton procedure we use the adaptive time stepping. If the algorithm for one time step does not converge we recompute it with a smaller time step (typically $\Delta t_{new} = 0.8 \Delta t_{old} $). For the Newton and the linearization methods the factorization is recomputed for each time step and during a Newton step the factorization is recomputed if the number of GMRES iterations associated with the two last Newton steps is superior to 50.

\subsection{Model without parallel velocity}
\subsubsection{First test case }
This first case corresponds to a simplified equilibrium configuration associated to the JET reactor.
We solve the model without parallel velocity. In this case the numerical viscosity  is zero and the numerical resistivity is $10^{-10}$. The physical viscosity and resistivity, dependent on the temperature are given by $\eta(T)=2 \times 10^{-6} T^{-\frac52}$ and $\nu(T)=4 \times 10^{-6} T^{-\frac52}$. Note that the energy estimate for our model is valid only for a constant viscosity coefficient, which is not the case here. This point will be discussed in the future. We consider a geometry with X-point. The number of degrees of freedom for these simulations is around $1.5\times 10^5$ with around $9.0\times 10^7$ nonzero coefficients. In the toroidal direction we use three Fourier modes $1$, $\operatorname{cos}(n_p \phi)$ and $\operatorname{sin}(n_p\phi)$ with $n_p$ a parameter called the periodicity.
 For the linearization procedure the criterion of convergence for the GMRES procedure is $\eps=10^{-8}$. 
For the Newton procedure the maximum number of Newton iteration is 10 and the criterion of convergence for the Newton procedure $\eps_{a}=10^{-5}$, the $\eps_{GMRES}^0$ of the GMRES convergence  criterion is $0.0005$.  Using $\Delta t = 30$ we compare the results for the linearization method, the exact Newton method and the inexact Newton method. These results are given between the time 1400 and 3500 corresponding to the nonlinear saturation phase. The final time is $3500$. The code is executed with 2 MPI and 16 OpenMP threads per MPI process. In the tables (Tab 1) - (Tab 2) - (Tab 3) we give the average of different quantities associated to the solver during one time step.

\begin{table}[h]
\begin{tabular}{| c ||  c| c| c|}
\hline
\multicolumn{4}{| c |}{Linearization method}\\
\hline
$\Delta t=$ & GMRES Iter. & $LU$ fact. & time\\
\hline
30 &  19  & 1 &  53.25\\
\hline
\end{tabular}
\caption{For the linearization method the average number of GMRES iterations and $LU$ factorizations per time step are given as well as the wall clock time}
\end{table}

\begin{table}[h]
\begin{tabular}{| c ||  c| c| c| c| c|}
\hline
\multicolumn{6}{| c |}{Exact Newton}\\
\hline
$\Delta t=$ & GMRES Iter. & $LU$ fact. & Newton iter. & Total GMRES iter. & time\\
\hline
30 & 19.8 &  1.05  & 3 &  59 & 79.6\\
\hline
40 & 26.6 & 1.28 & 3.2 & 85.5 & 102\\
\hline
\end{tabular}
\caption{For the exact Newton method the average number of total GMRES iterations, $LU$ factorizations, number of Newton iteration and number to GMRES iteration per Newton step per time step are given as well as the wall clock time.}
\end{table}

\begin{table}[h]
\begin{tabular}{| c ||  c| c| c| c| c|}
\hline
\multicolumn{6}{| c |}{Inexact Newton}\\
\hline
$\Delta t=$ & GMRES Iter. & $LU$ fact. & Newton iter. & Total GMRES iter. & time\\
\hline
30 & 3.3 &  1  & 5.7 &  18.7 & 76.25\\
\hline
40 & 5.4 & 1 & 5.8 & 31 & 82.9\\
\hline
\end{tabular}
\caption{For the inexact Newton method the average number of total GMRES iterations, $LU$ factorizations, number of Newton iteration and number to GMRES iteration per Newton step per time step are given as well as the wall clock time.}
\end{table}

Some remarks about these results. It is clear that the CPU cost associated with the Newton procedure is higher compared to the classical linearization for the same time stepping. This result is expected, indeed by definition of the Newton method, the number of linear problems solved is larger with the Newton procedure.
First in the nonlinear phase  we remark that the Newton procedure is also less performing, but using an inexact Newton method we can reduce the  CPU cost. In the Tables 2 and 3 we remark that for the time step $\Delta t =30$, the main difference between the Inexact and exact Newton method is small, but for $\Delta t=40$ the difference is larger. The  main difference between exact and inexact Newton method can be explained by the fact that the number of GMRES iterations is larger using the exact Newton method and  consequently the factorization for the preconditioning is called more often. At the end we remark that the inexact Newton method is clearly more efficient when the problem is more nonlinear and similar when the problem is not too nonlinear. This result verifies the usefulness of the inexact Newton method.

Now we propose to compare the linearization method and the inexact Newton method. In the nonlinear phase the difference is less important. Indeed in the nonlinear phase the number of GMRES iterations for each linear problem is larger. Using an inexact Newton procedure we have more linear problems to solve but each linear system is solved with a small accuracy. Consequently the cost associated with each linear system is smaller when we use the inexact Newton procedure. The Tables 1 - 3 for $\Delta t =30$ show that the total GMRES iterations for one time step are similar between an inexact Newton procedure and the linearization procedure. Consequently the additional cost associated with the inexact Newton method come from the computation of the matrix and in this case is around 1.5 which is an acceptable additional cost. Additionally, the parallel scaling is better for the construction of the matrix than the iterative solver and the preconditioning. Consequently with more MPI process the difference between the CPU cost associated with the inexact Newton method and the linearization method can be reduced. \\

Secondly we compare the two methods with $\Delta t =40, 50, 60$. For the Newton procedure the maximum number of Newton iterations is 20 and the criterion of convergence for the Newton procedure $\eps=10^{-7}$, the $\eps_0$ of the GMRES convergence  criterion is $0.0005$. We plot the kinetic and magnetic energies associated with the different modes for the two procedures and the different time steps (Fig 3.) and (Fig 4.).

\begin{figure}[t]
\begin{center}
    \includegraphics[scale=.6]{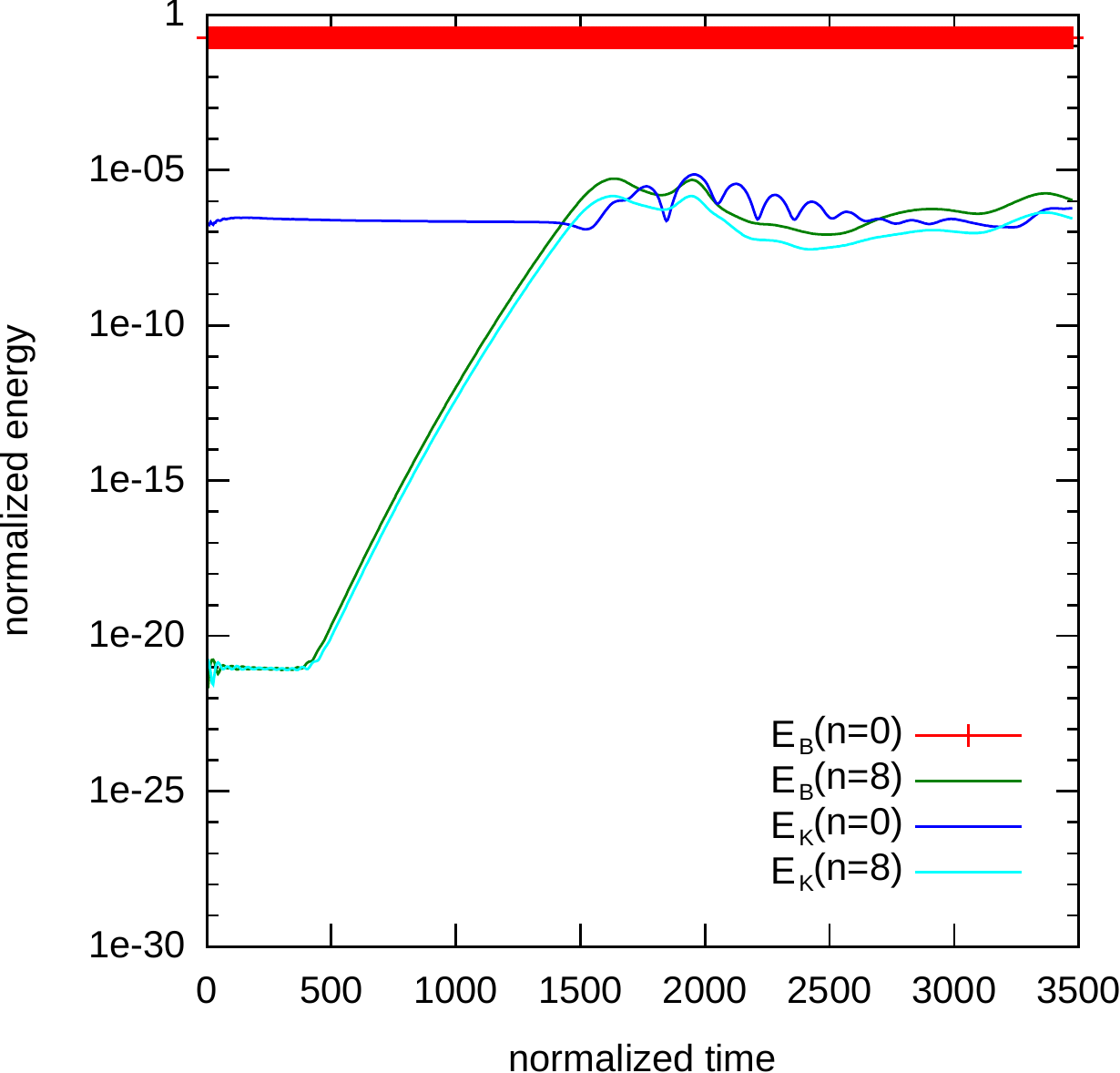}
    \end{center}
  \caption{Kinetic and Magnetic energies for $\Delta t=5$ gives by the Newton method. This solution can be use as reference to validate the solutions computed with very large time steps. The coefficient $n$ corresponds to the periodicity. In this case the periodicity is equal to $8$.}
\end{figure}

\begin{figure}[t]
\begin{center}
    \includegraphics[scale=.5]{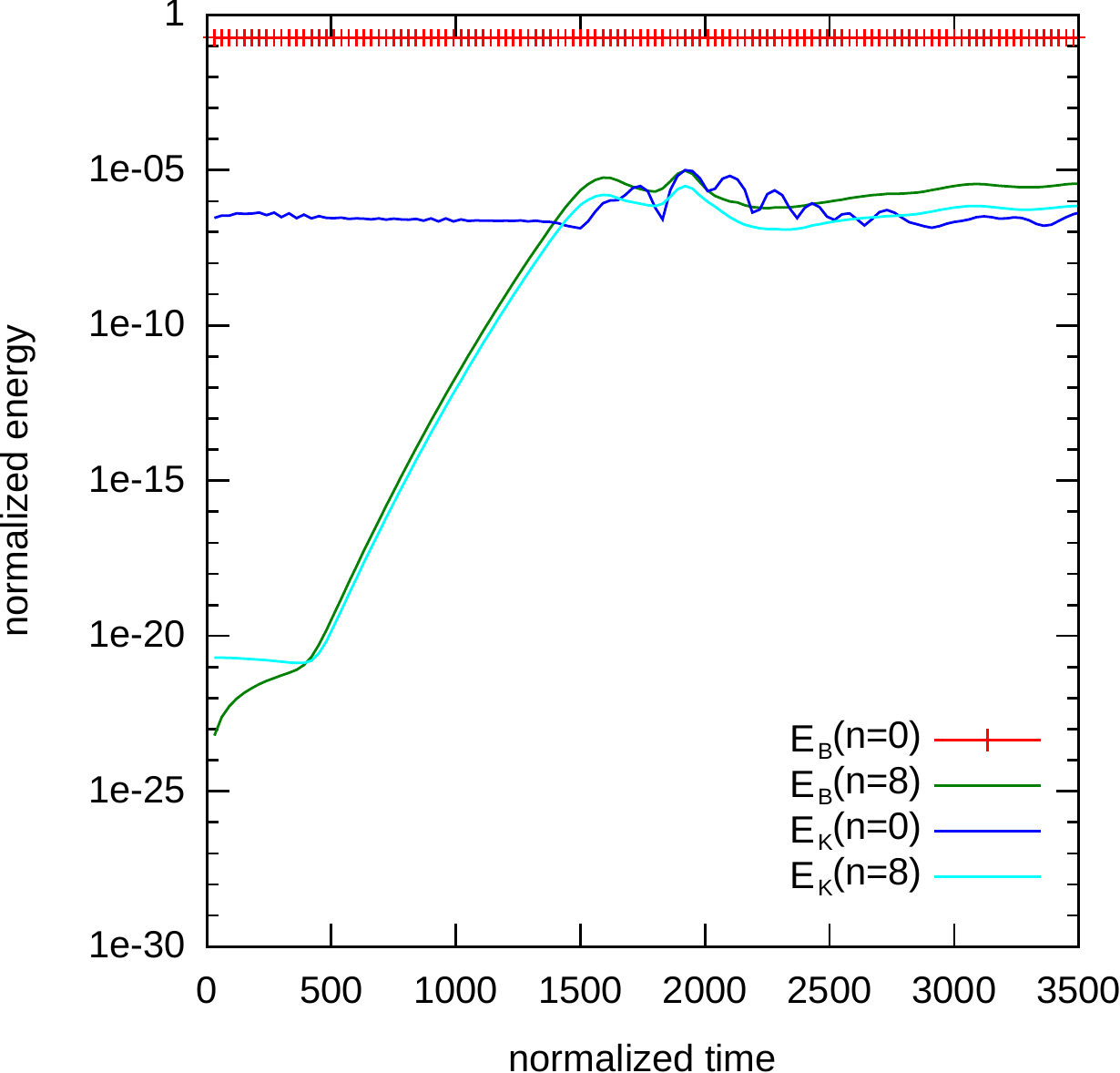}\hspace{40pt}\includegraphics[scale=.5]{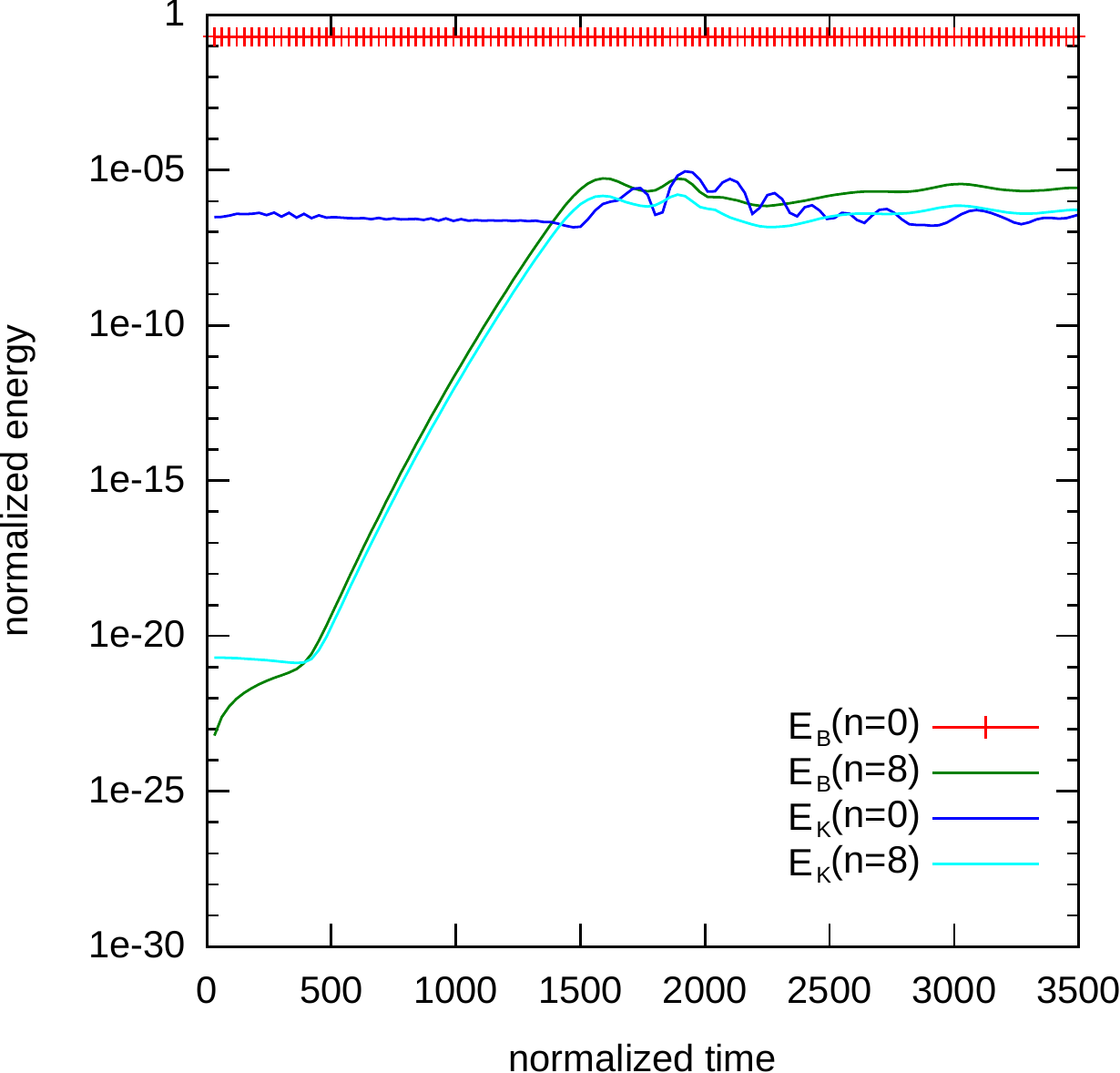}
    \end{center}
\begin{center}
    \includegraphics[scale=.5]{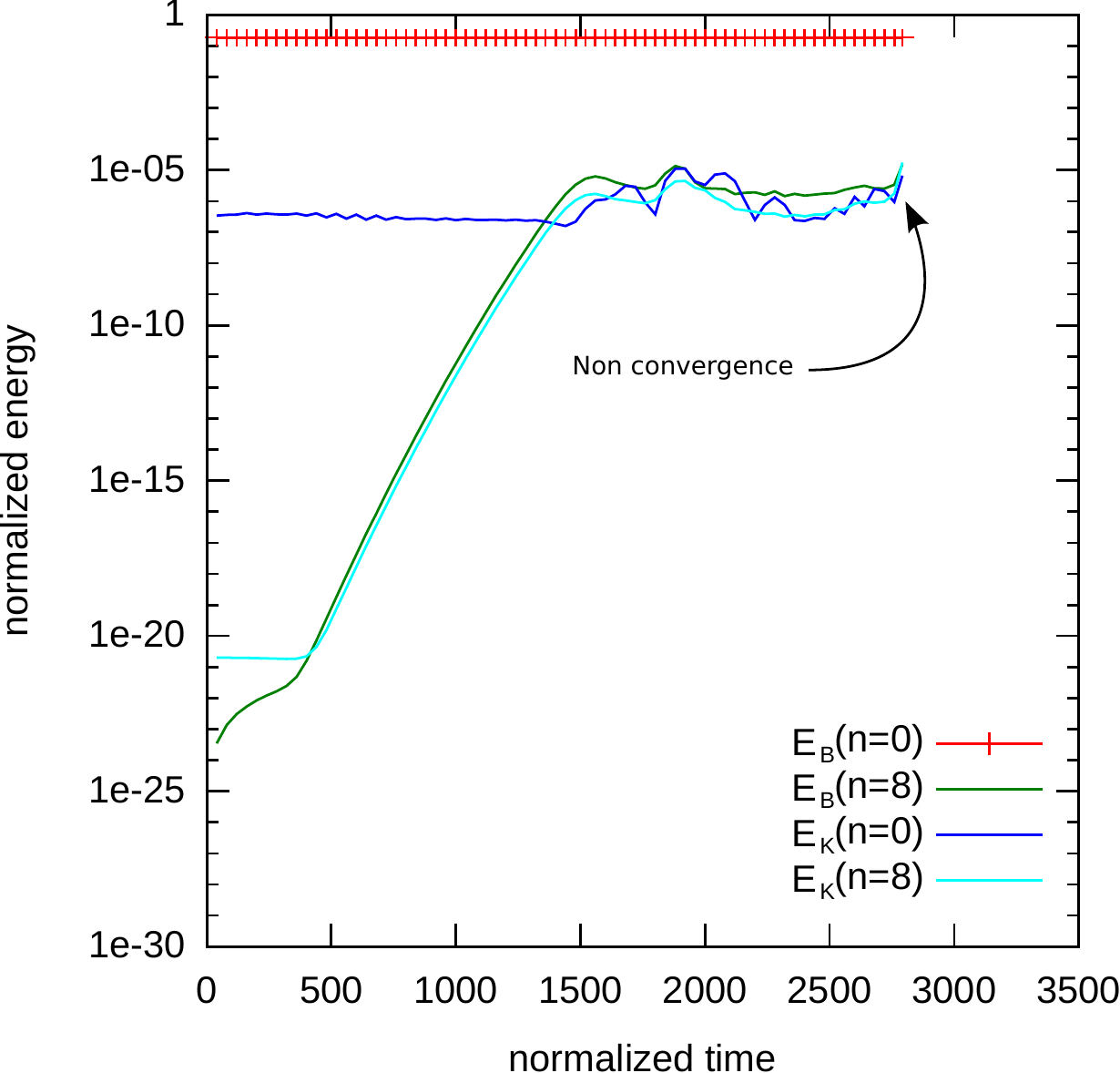}\hspace{40pt}\includegraphics[scale=.5]{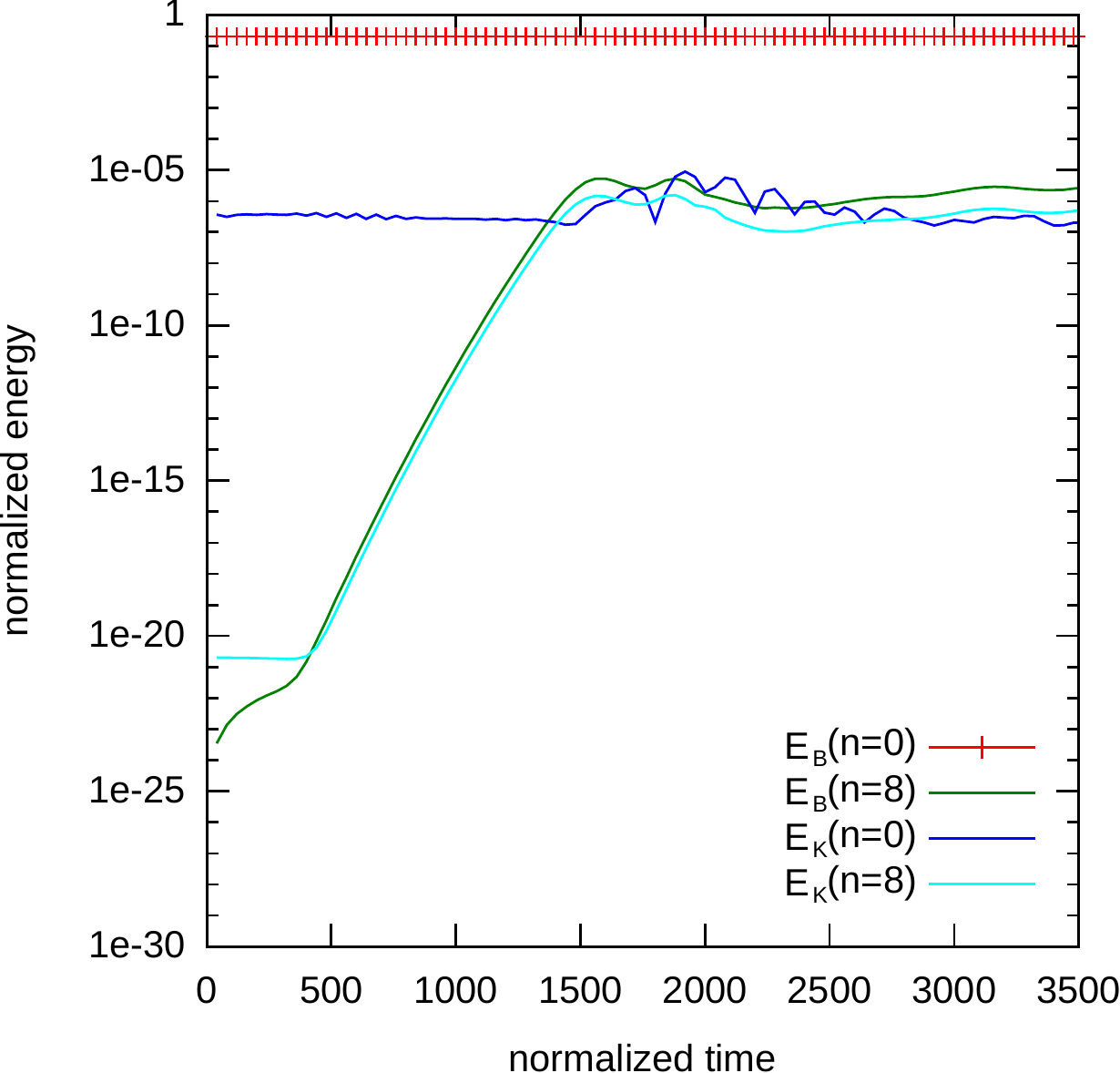}
\end{center}
\begin{center}
    \includegraphics[scale=.5]{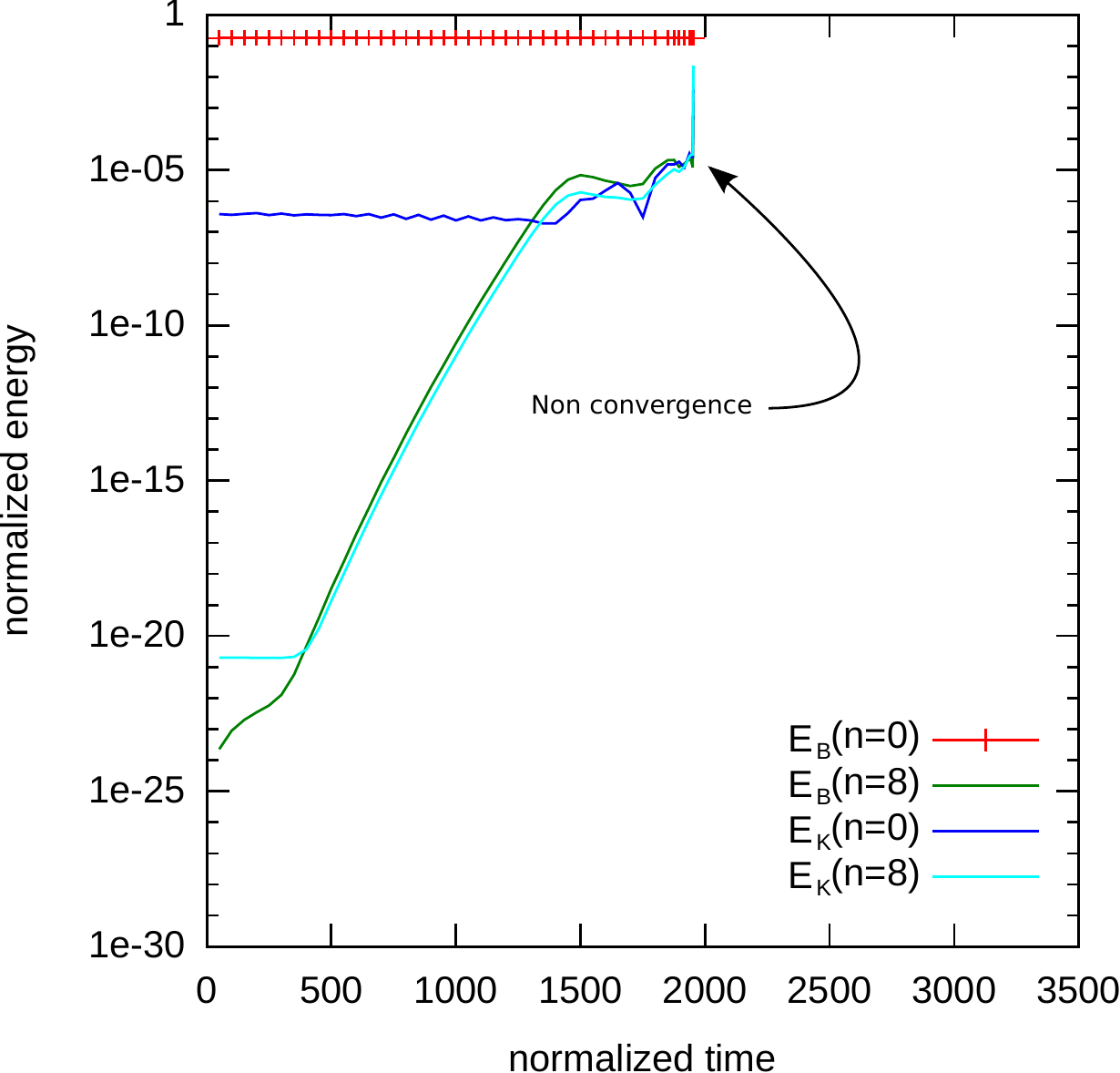}\hspace{40pt}\includegraphics[scale=.5]{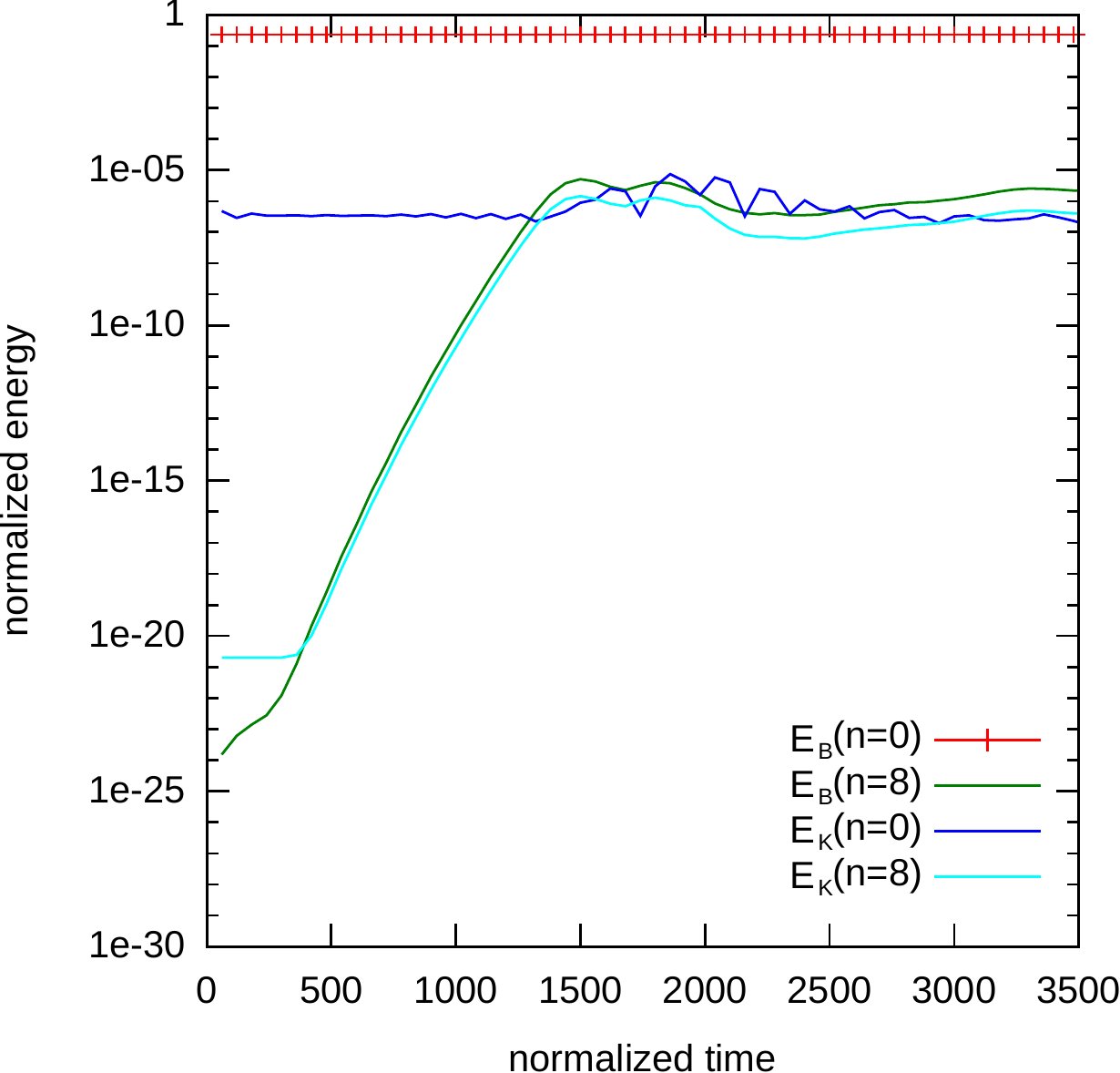}
\end{center}
  \caption{In the left Kinetic and Magnetic energies for Linearization method for $\Delta t=30,40, 50$. In the left Kinetic and Magnetic energies for Newton method for $\Delta t=30, 40, 60$}
\end{figure}

These results show that the Newton procedure with adaptive time stepping is more robust than the classical linearization and allows to use a larger time step.
When we use the classical linearization  with a very large time step, the numerical error linked to the time discretization and the linearization becomes too large such that consequently the numerical scheme does not capture correctly the beginning of the saturation phase. In this case, numerical instabilities appear and the iterative solver does not converge after the beginning of the numerical instability. If we use an adaptive time stepping the situation is the same because in general the scheme computes the beginning of the numerical instabilities and at this moment is too late to adapt and decrease the time step.

With the Newton procedure the situation is different. First the error of linearization and consequently the global numerical error is smaller so we can use larger time steps and capture correctly the beginning of the saturation phase. Secondly we don't have the problem associated with the numerical instabilities with the Newton procedure as the Newton method does not converge in case of the numerical instability such that the time step is recalculated with smaller $\Delta t$. We conclude that the adaptive time stepping works with the Newton method because this procedure detects the beginning of the numerical instabilities by non convergence of the method contrary to the linearization, for which in order to continue the computation it is necessary to adapt the $\Delta t$ before the beginning of the  numerical instability.
 Consequently the Newton procedure is more robust, allowing an efficient adaptive time stepping, which avoids numerical instabilities for large time steps and non convergence issues.
 The figure (Fig 4.)  shows that the code with the linearization method does not converge with $\Delta t=40$ contrary to the Newton method which converges even with $\Delta t=60$. 
 This test case is not too nonlinear and consequently not too stiff for the numerical method. For more nonlinear test cases the Newton procedure gives better results when the problems get stiffer.

\subsubsection{Second test case }
This second test case corresponds to a realistic ASDEX Upgrade equilibrium configuration with unrealistically large resistivity which makes the instability especially violent.
We solve the model without parallel velocity. In this case the numerical viscosity and the numerical resistivity are close to $10^{-11}$. The physical viscosity and resistivity are dependent of the temperature : $\eta(T)=2 \times 10^{-5} T^{-\frac52}$ and $\nu(T)=3 \times 10^{-5} T^{-\frac52}$. We consider a geometry with X-point. In the toroidal direction we use three Fourier modes $1$, $\operatorname{cos}(n_p \phi)$ and $\operatorname{sin}(n_p\phi)$ with $n_p$ a parameter called the periodicity. The final time is $450$.\\

For the linearization procedure the maximum number of GMRES iteration is 500 and the convergence criterion  for the GMRES procedure $\eps=10^{-8}$. 
First we propose to compare the two methods  for $\Delta t =5, 10, 20$. For the Newton procedure the maximum number of Newton iterations is 10 and the the criterion of convergence for the Newton procedure $\epsilon=10^{-5}$, the $\eps_0$ of the GMRES convergence  criterion is $0.0005$. \\

\begin{figure}[t]
\begin{center}
    \includegraphics[scale=.58]{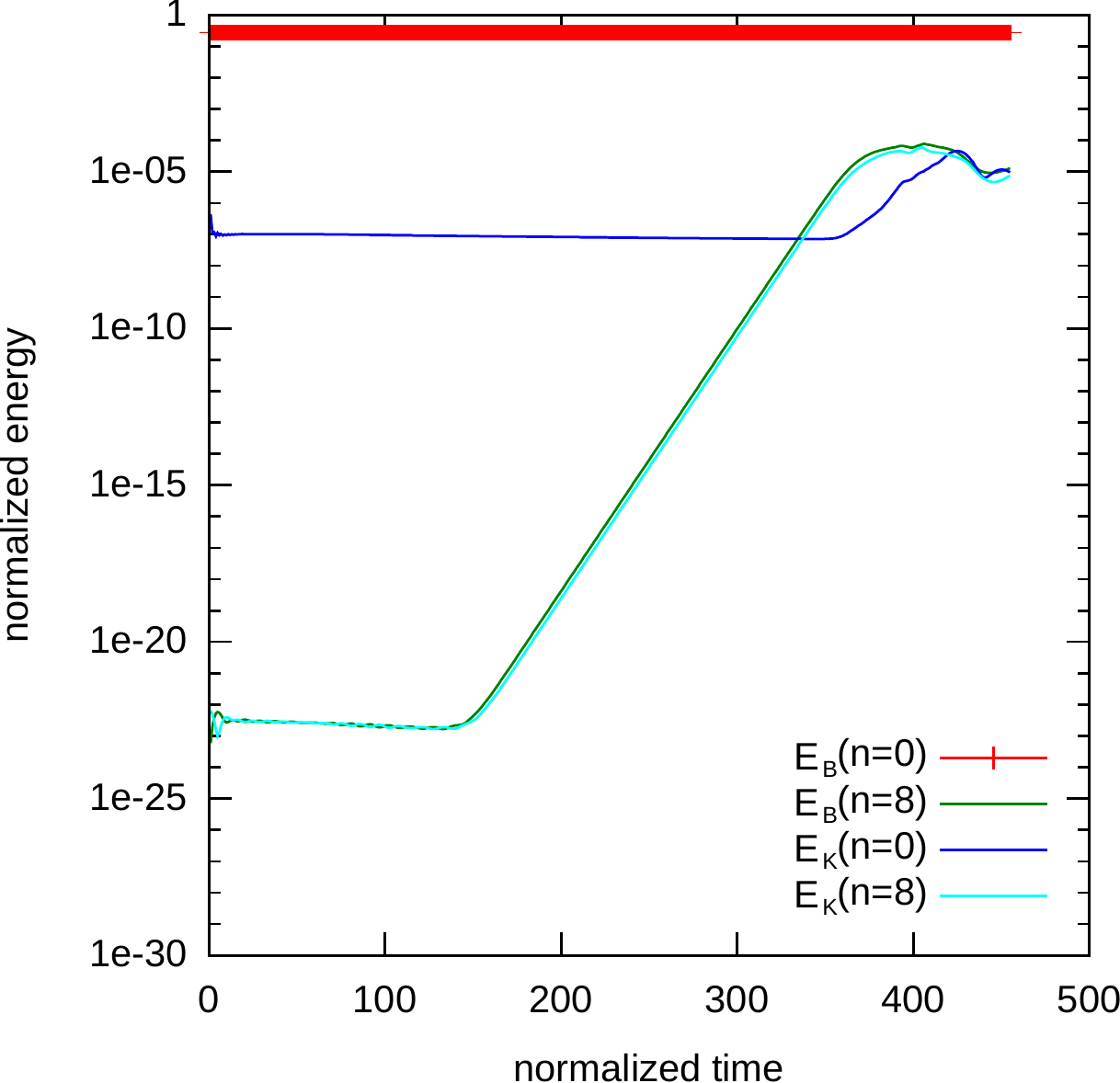}\hspace{10pt}\includegraphics[scale=.58]{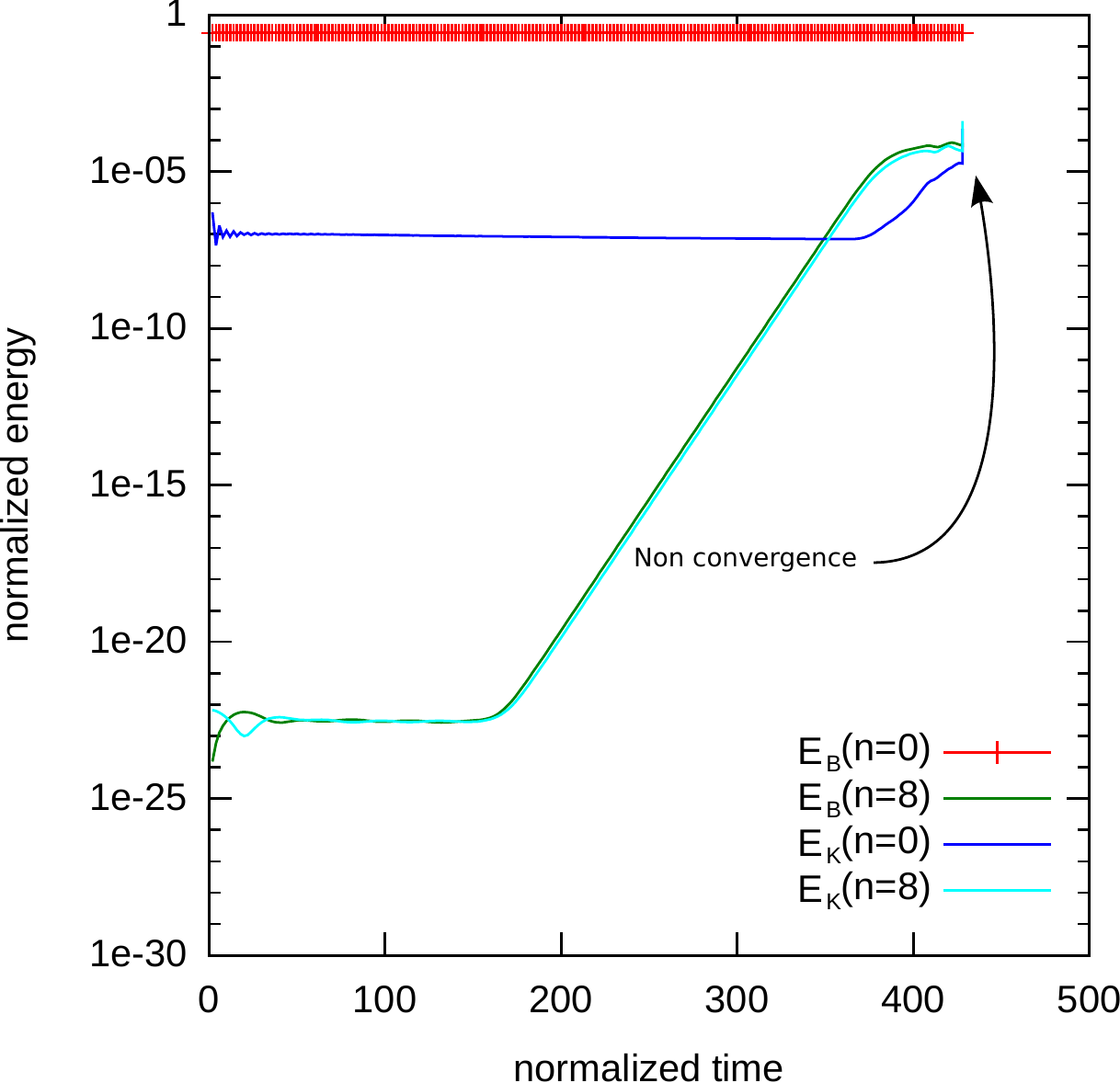}
 \end{center}
 \begin{center}   
    \includegraphics[scale=.58]{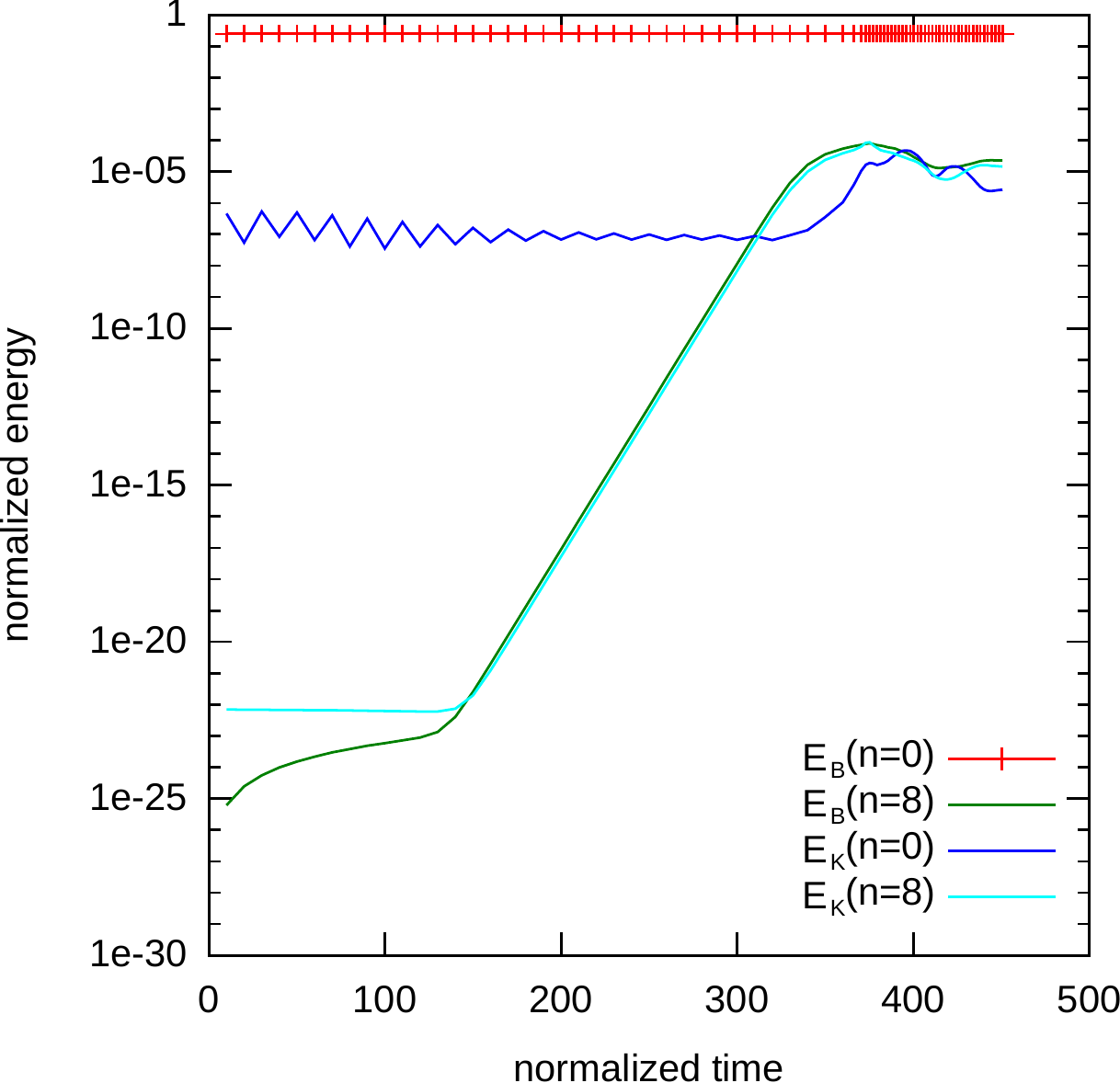}
\end{center}
  \caption{In the left Kinetic and Magnetic energies for Linearization method for $\Delta t=1$. In the middle Kinetic and Magnetic energies for Linearization method for $\Delta t=2$. In the right Kinetic and Magnetic energies for the Newton method for $\Delta t=10$ with adaptive time stepping.}
\end{figure}

This test case with violent physical instabilities allows to confirm the previous remarks about adaptive time stepping and numerical instabilities. First if we choose a too big time step with the linearization method, we have a numerical instability which appears and the adaptive time stepping is not efficient. 

\begin{table}[h]
\begin{tabular}{| c | c |  c | c |  c| }
\hline
 GMRES $\Delta t=1$  &  Newton + adaptive time method \\
\hline
 18800 & 7600 \\
 \hline
\end{tabular}
\caption{CPU time of the simulation for the GMRES method with $\Delta t=1$ and for Newton method with adaptive time method (initial time step $\Delta t=10$).}
\end{table}

Due to the violent physical instabilities the problem is strongly nonlinear in the saturation phase.  Contrary to the previous test case, using the Newton method allows to reduce significantly the CPU cost for the total run (Tab 5.).\\
The inexact Newton method with adaptive time stepping is more robust than the linearization method and allows to reduce the CPU costs for highly nonlinear cases because of the possibility to take larger time steps.

\subsection{Model with parallel velocity}
This test case is the same as the one used in section 4.1.1 but we solve  the model with parallel velocity.
First we compare the two methods in the nonlinear phase with $\Delta t =20$. For the Newton procedure the maximum number of Newton iteration is 10 and the the criterion of convergence for the Newton procedure $\eps=10^{-5}$, the $\eps_0$ of the GMRES convergence  criterion is $0.0005$.  These results are given between the time 1250 and 3500  which correspond to the saturation phase (stiff part of the computation).

\begin{table}[h]
\begin{tabular}{| c ||  c| c| c|}
\hline
\multicolumn{4}{| c |}{Linearization method}\\
\hline
models & GMRES Iter. & Facto. & time\\
\hline
with neglected terms &  25.4  & 1 &  75.7\\ 
\hline
without neglected terms  & 28 & 1 & 83.6 \\
\hline
\end{tabular}
\caption{Average of number of GMRES iteration and factorization computation (preconditioning) during a time step.  Linearization method. $\Delta t=20$}
\end{table}

\begin{table}[h]
\begin{tabular}{| c ||  c| c| c| c| c|}
\hline
\multicolumn{6}{| c |}{Inexact Newton method}\\
\hline
models & GMRES Iter. & Facto. & Newton iter. & Total GMRES iter. & time\\
\hline
with neglected terms & 5.1 & 1  & 6.4 & 32.7 & 119.3\\
\hline
without neglected terms  & 5.2 & 1 & 6.4 & 33.4 & 122.5\\
\hline
\end{tabular}
\caption{Average of number of GMRES iteration and factorization computation (preconditioning) during a time step.  Inexact Newton method. $\Delta t=20$}
\end{table}

\begin{table}[h]
\begin{tabular}{| c ||  c| c| c| c| c|}
\hline
\multicolumn{6}{| c |}{Inexact Newton method}\\
\hline
models & GMRES Iter. & Precon. called & Newton iter. & Total GMRES iter. & time\\
\hline
with neglected terms & 10.9 & 1.1  & 6.95 & 75.6 & 152\\
\hline
without neglected terms  & 8.7 & 1 & 6.7 & 58 & 142\\
\hline
\end{tabular}
\caption{Average of number of GMRES iteration and preconditioning called during a time step.  Inexact Newton method. $\Delta t=20$}
\end{table}

 The conclusions on the comparison between the Newton procedure and the linearization procedure are similar to the conclusion for the first test case: in the nonlinear phase the new method costs around 1.4 - 1.5 times more, but this additional cost can be reduced using a larger time step. Indeed using the Newton procedure (as previously) we can use larger time steps  than with the original linearization method without running into numerical instabilities. For example in this latter case the linearization method is unstable with $\Delta t =25 $ and the Newton method is stable with $\Delta t =40$ (Fig 6.). 

\begin{figure}[t]
\begin{center}
    \includegraphics[scale=.53]{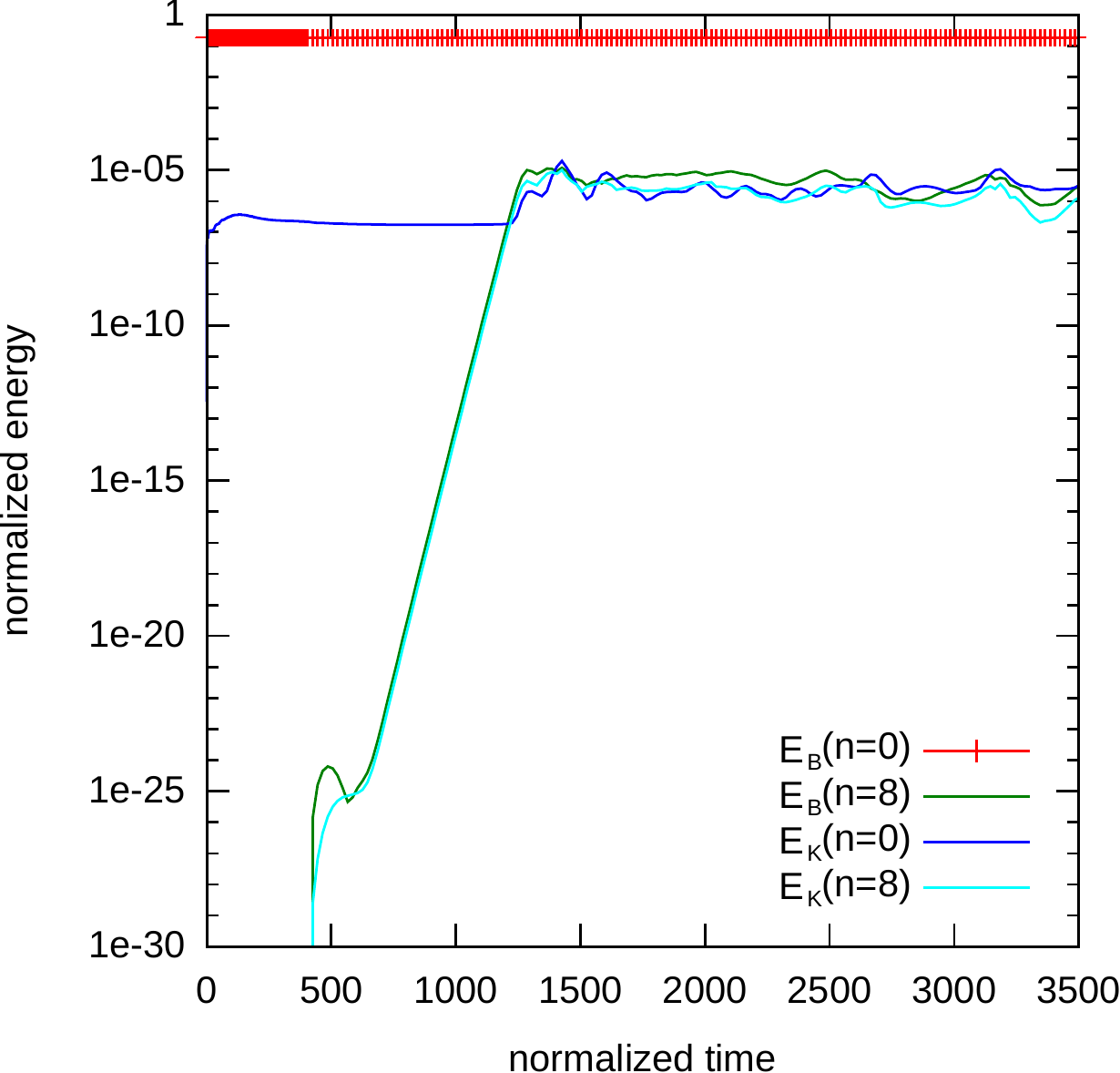}\hspace{30pt}\includegraphics[scale=.53]{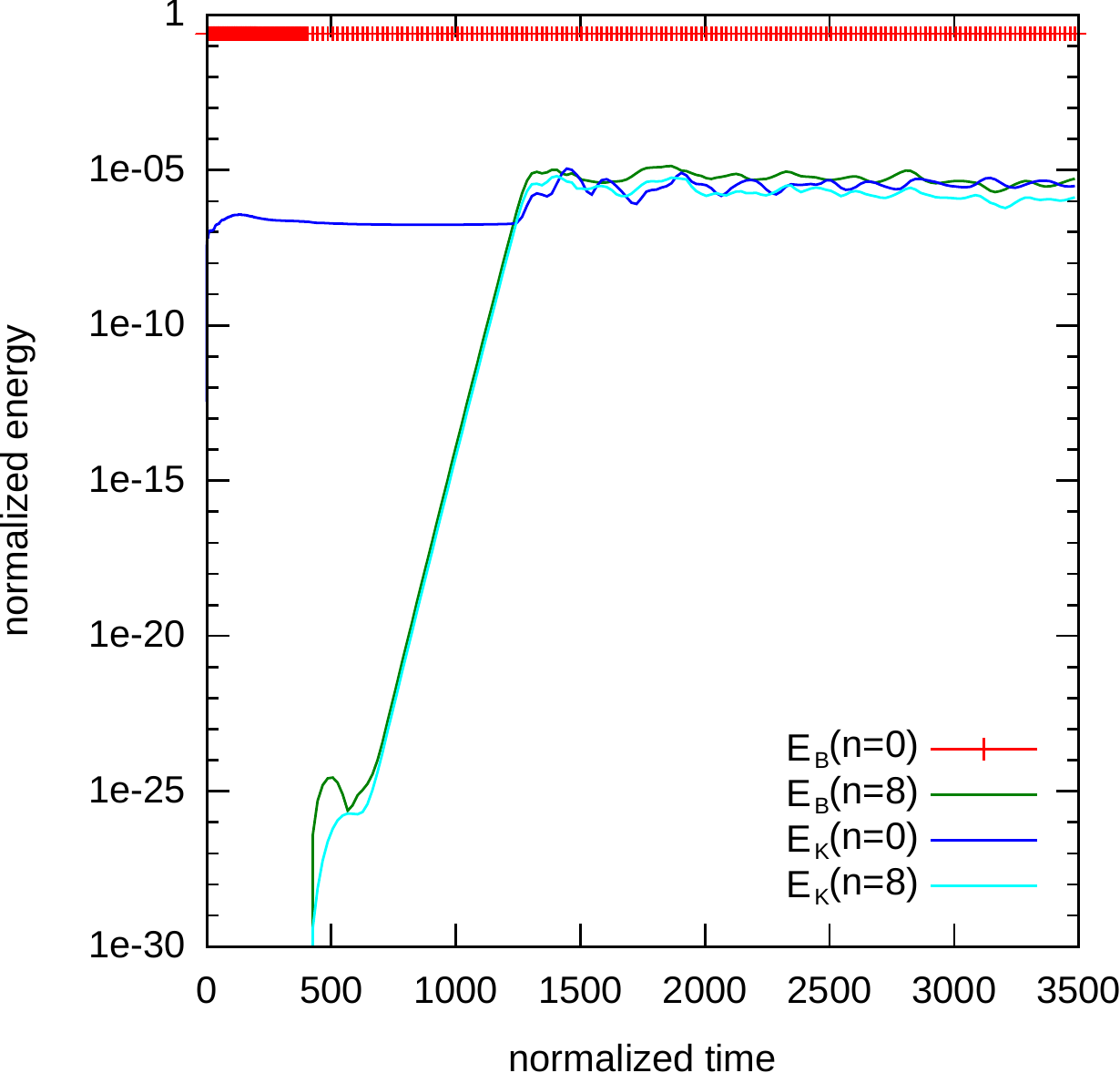}
    \end{center}
\begin{center}
    \includegraphics[scale=.53]{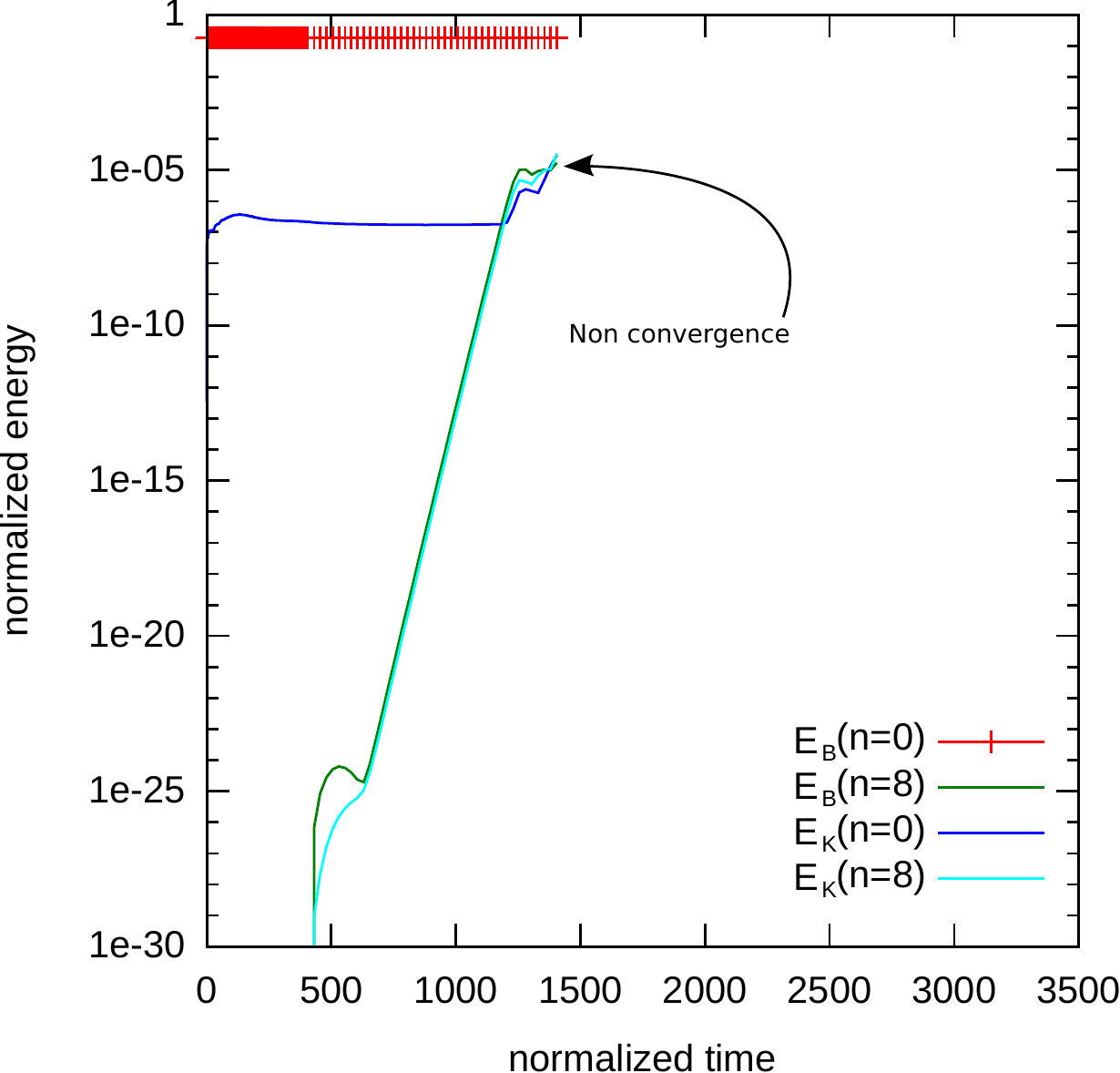}\hspace{30pt}\includegraphics[scale=.53]{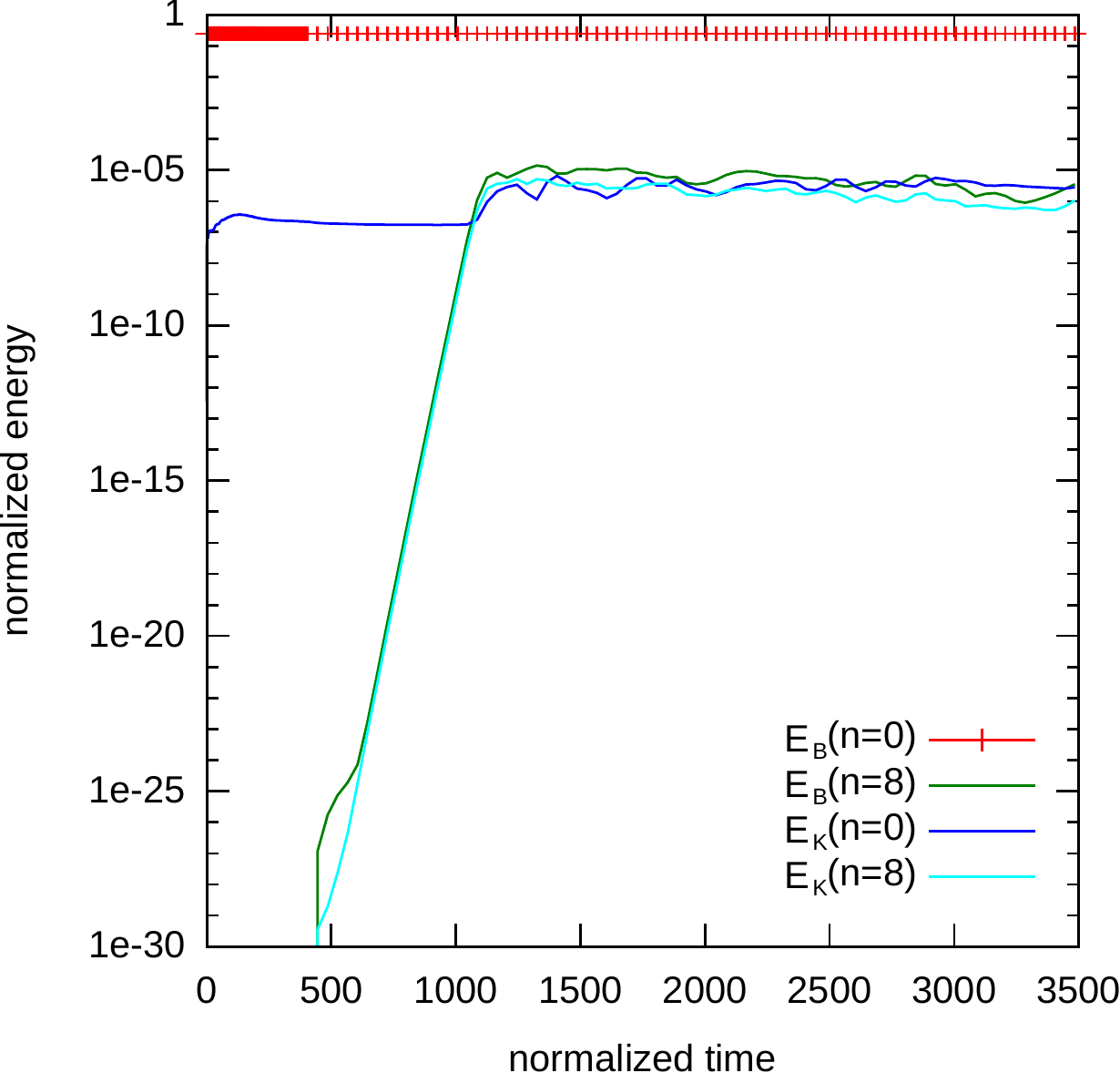}
\end{center}
  \caption{In the left Kinetic and Magnetic energies for Linearization method for $\Delta t=20,25$. In the left Kinetic and Magnetic energies for Newton method for $\Delta t=20, 40$}
\end{figure}

 For all these results we recompute the factorization for the preconditioning at each time step. For the Newton method we have added an additional rule. The factorization is recomputed if the convergence is too slow for the previous linear step. To reduce the CPU time we can use only the second rule for the Newton procedure and the Linearization method. In this case, it is not necessary to compute the factorization for each time step. The different test cases show that for the Newton procedure it will be important to use a smaller $\eps_0$ (initial $\eps$ for the GMRES method in the Inexact Newton procedure) to compute correctly the first Newton iteration.
 
 The last remark about this result is on the difference between the model with and without neglected terms. These terms in the potential and parallel velocities equations come from to the fact  the  poloidal and parallel velocity are not perpendicular, this is the cross terms between the poloidal velocity and the poloidal part of the parallel velocity. In the (Fig 7.) we remark that we have small differences in the dynamics of kinetic and magnetic energies  between the models with and without neglected terms. We observe these differences for the linearization method with $\Delta t=20 $ and for the Newton methods with $\Delta t=40$. With the Newton procedure and $\Delta t=20$ the difference is smaller. In theory these terms are small consequently it is expected that the impact of these terms is small when the numerical error (Time and linearization errors) is small. When the error is larger (Linearization method with $\Delta t=20$, Newton method with $\Delta t=40$) the impact of these terms is more important.  However the impact of these additional terms on the stability, conditioning  and convergence issues is not clear and requires additional studies for exemple when the resistivity and viscosity are close to zero.
 \begin{figure}[t]
\begin{center}
    \includegraphics[scale=.58]{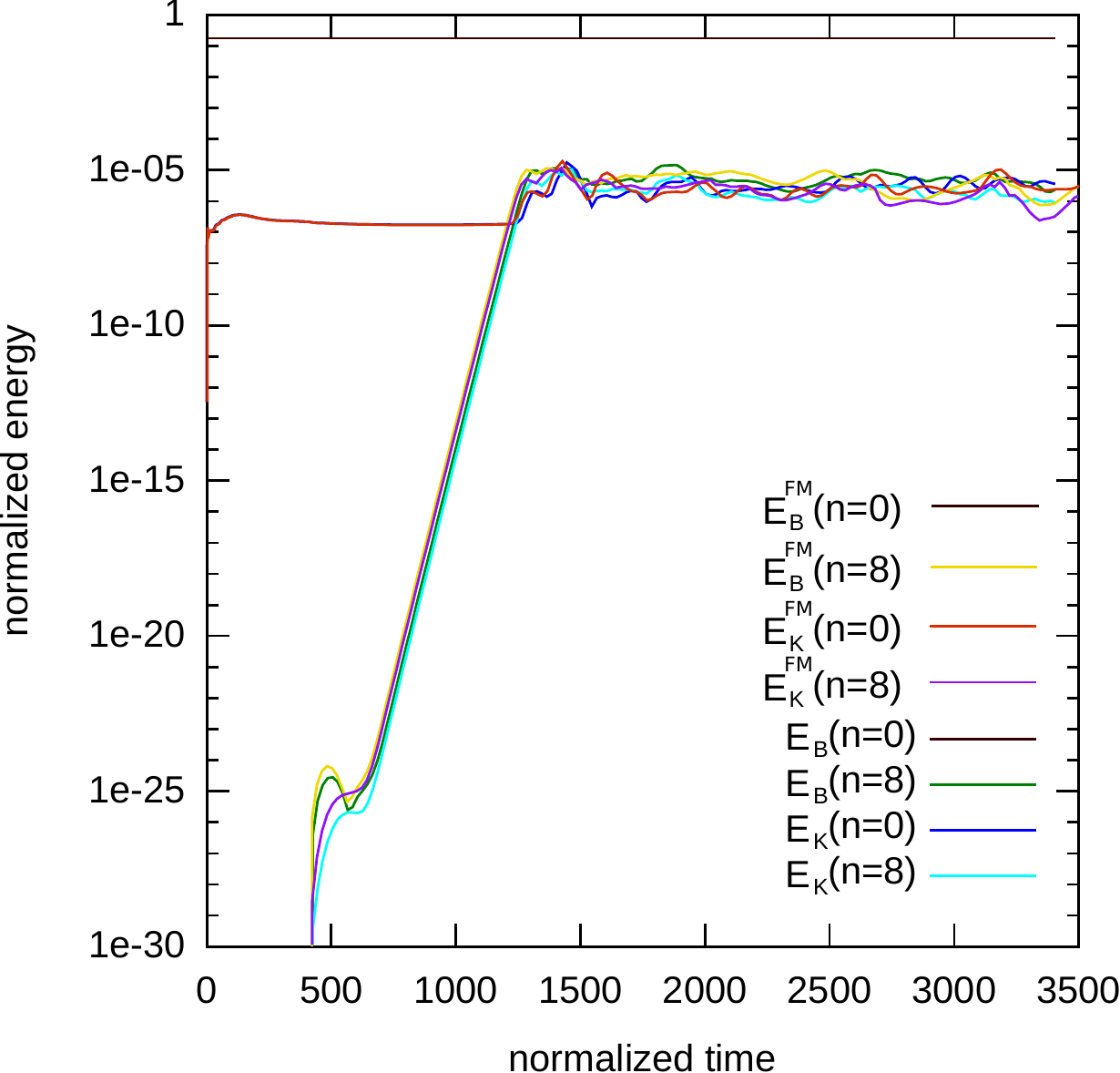}\hspace{10pt}\includegraphics[scale=.53]{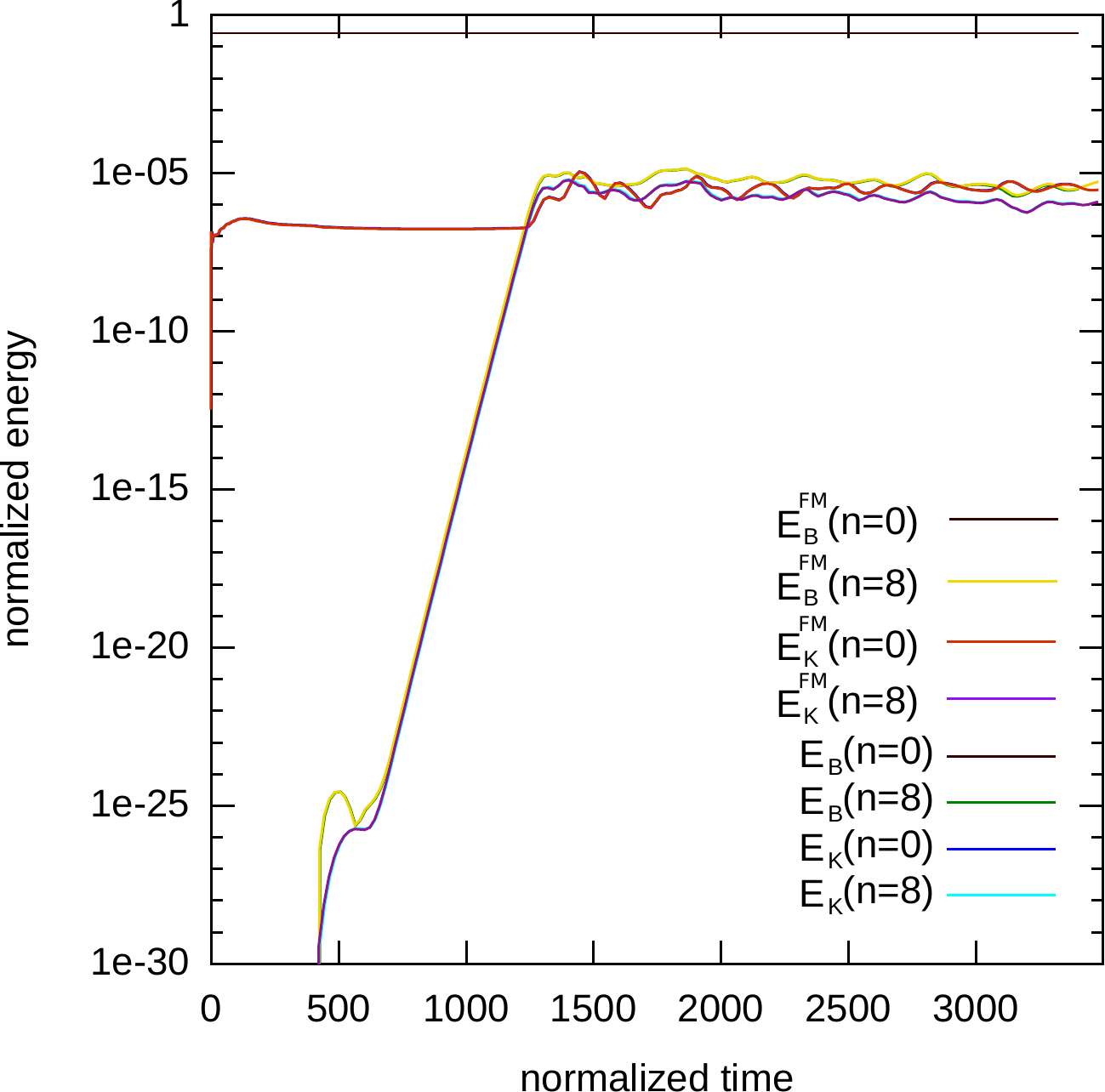}\hspace{10pt}
    \end{center}
    \begin{center}
    \includegraphics[scale=.58]{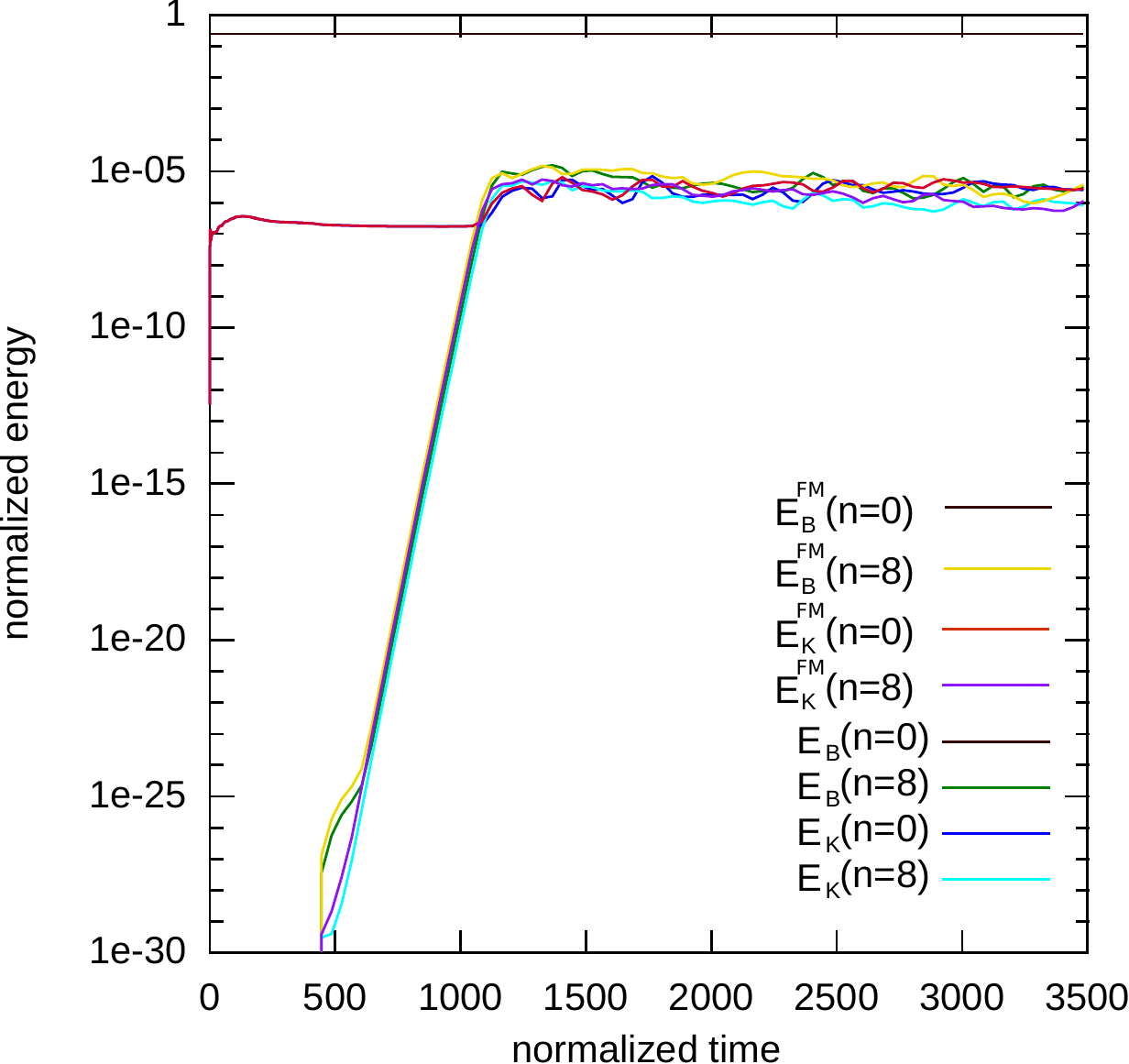}
\end{center}
  \caption{Comparaison between the full model (model with neglected terms) and the model without neglected terms. In the top and left results given by the Linearization method with $\Delta t=20$, in the top and right, results given by the Newton method with $\Delta t=20$ and in bottom results given by the Newton method with $\Delta t=40$.}
\end{figure}

\section{Conclusion}
In this paper, we have presented a rigorous analytical derivation of
the reduced MHD models implemented in the non-linear MHD code JOREK.
Starting from the potential formulation of the magnetic field vector
and fluid velocity used in JOREK we obtain a few additional terms that
have been neglected in the code but might be relevant in the
non-linear phase. We have also given a proof of the conservation (in the ideal case) or dissipation (in the resistive and viscous cases) of
total energy for this reduced MHD model if the additional terms are
taken into account. This is an important validation for the choices of
the projections and the assumptions of the derivation. Indeed we obtain an energy estimate close the energy estimate associated with the full MHD.  At the numerical level it is important to have a stable model (here we consider the dissipation of the energy as a first stability result  for the model). Indeed it is not possible to certify the numerical stability (decay of the norm or of  the energy) and the good behavior of the numerical methods if it is not the case for the continuous model. The numerical results do not show large differences between our model derived previously with the dissipative energy estimate and the model implemented initially in JOREK  which does not preserve this energy balance estimate. Perhaps because at the numerical level the discrete energy decay is not yet satisfied exactly under all circumstances (time scheme not adapted, negative density or wrong viscous coefficient can  explain this). However, now we have a model with a good energy  balance law  which makes possible the design of numerical method that are stable and robust. In the future we would like to study the derivation of the reduced MHD with the bi-fluid effects, with more physical stress tensors \cite{ionplasmas}.\\\\
The second part of the paper is focused on the time solver of JOREK.
The original method used in JOREK for the time-stepping of the
nonlinear system is a linearization solved iteratively by GMRES with
physics-based preconditioning. We have replaced this by the nonlinear
inexact Newton method in which the linear convergence accuracy of
GMRES depends on the non-linear convergence. Especially at the onset of
non-linear saturation, large numerical errors can cause numerical
instabilities and prevent convergence. The non-linear time stepping
reduces those errors and consequently allows to use larger time steps
as confirmed by numerical tests. We have also implemented and tested
an adaptive time stepping that works very efficiently with the Newton
method and allows to reduce computational costs. The Newton method is
more robust than the linearization method as it avoids certain
numerical instabilities, is well suited for adaptive time stepping,
and allows to reduce computational costs in highly non-linear cases.
The Newton method is currently implemented for the single fluid
reduced MHD equations in JOREK, and will be extended to two-fluid
terms and further extended models in the future.

\end{document}